\documentclass[english]{amsart}
\usepackage[T1]{fontenc}
\usepackage{euscript}
\usepackage{textcomp}
\usepackage{subfig}
\usepackage{amsmath}
\usepackage{amssymb,mathrsfs,hyperref,url,bbm}
\usepackage{amsthm}
\usepackage[english]{babel}
\usepackage{graphicx}
\theoremstyle{definition}
\newtheorem{definition}{Definition}[section]
\newtheorem{remark}{Remark}[section]
\theoremstyle{theorem}
\newtheorem{theorem}{Theorem}[section]
\newtheorem{lemma}{Lemma}[section]
\newtheorem{proposition}{Proposition}[section]
\newtheorem{corollary}{Corollary}[section]

\def\rg{\mathscr{R}}
\def\slbf#1{\text{\boldmath$#1$}}
\def\boundaries{\mathscr C}
\def\shift{\mathscr S}

\title[Dynamical systems around the Rauzy gasket]{Dynamical systems around the Rauzy gasket and
their ergodic properties}
\author{Ivan Dynnikov}
\address{\noindent Steklov Mathematical Institute of Russian Academy of Science, 8 Gubkina Str., Moscow 119991, Russia}
\email{dynnikov@mech.math.msu.su}
\author{Pascal Hubert}
\address{Institut de Math\'ematiques de Marseille, 39 rue F. Joliot-Curie, 13453 Marseille Cedex 20, France}
\email{pascal.hubert@univ-amu.fr}
\author {Alexandra Skripchenko}
\address{Faculty of Mathematics, National Research University Higher School of Economics, 6 Usacheva Str., 119048 Moscow, Russia \textit{and}
Skolkovo Institute for Science and Technology, Skolkovo Innovation Center, 143026 Moscow, Russia}
\email{sashaskrip@gmail.com}

\begin{document}
\begin{abstract}
At the beggining of the 80's,  H.\,Masur and W.\,Veech started the study of generic properties of interval exchange transformations proving that almost every such transformation is uniquely ergodic.
About the same time, S.\,Novikov's school
and French mathematicians
independently discovered very intriguing phenomena for classes of measured foliations on surfaces and respective IETs.
For instance,  minimality is exceptional in these families. A~precise version of this statement is  a conjecture by Novikov.  
The French and Russian constructions are very different ones.  Nevertheless, in the most simple situation (surfaces of genus three with two singularities) it was recently observed that both foliations share the same type of properties.  For instance, the space of minimal parameters is the same, called the Rauzy gasket. However, the precise connection between these two series of works was rather unclear.
The aim of this paper is to prove that both theories describe in different languages the same objects.  This text provides an explicit dictionary between both constructions.\end{abstract}

\maketitle

\tableofcontents

\section{Introduction}

It is known since the work of M.\,Keane~\cite{K} that interval exchange transformations (IET)
with irreducible permutation and parameters independent over~$\mathbb Q$ are minimal.
H.\,Masur~\cite{M} and W.\,Veech~\cite{V} have shown that almost all such transformations
are uniquely ergodic. However, if integral linear restrictions are
imposed on the parameters, then none of these results applies to the obtained family
of IETs, and there are families of this kind in which minimality and unique ergodicity become
exotic properties rather than generic ones.

In this paper, we discuss a family of IETs and respective singular measured foliations
on a surface of genus three. This family is apparently the simplest one among
those defined by linear restrictions on the parameters in which the subspace
of minimal (or uniquely ergodic) objects is a fractal subset. Several instances
of this family have been appeared previously in the literature, which we now overview.

The most celebrated example of a foliation from the class we are going to discuss
was described by P.\,Arnoux and J.-C.\,Yoccoz in 1981~\cite{AY}. They constructed a pseudo-Anosov homeomorphism of~$\mathbb RP^2$ and the stable foliation for it. Then they considered 
a ramified orientable  cover of~$\mathbb RP^2$ on which the foliation also became
orientable, and described this foliation in the language of interval exchange transformations.
The obtained IET was, naturally, self-similar in the sense that there was a rescaling
map conjugating this IET to its first return map to a smaller interval.
Among the seven parameters of the transformation only three were rationally independent.
(The construction of~\cite{AY} is for an arbitrary genus~$g\geqslant2$, but here we discuss only the case~$g=3$.)

There is a natural generalization of Arnoux--Yoccoz' construction, which gives rise to a family~$(\Sigma_\lambda,\mathscr F_\lambda)$ of surfaces of genus three endowed with a singular measured foliation
depending on three parameters~$\lambda=(\lambda_1,\lambda_2,\lambda_3)$.

In 1982 S.\,Novikov posed a question about asymptotic behavior of plane sections of
3-periodic surfaces in the three-space~\cite{N}. By a 3-periodic surface we mean a surface in~$\mathbb R^3$
invariant under shifts by all vectors from the lattice~$\mathbb Z^3$. The surfaces of interest
are also such that their projections to the 3-torus~$\mathbb T^3=\mathbb R^3/\mathbb Z^3$
are null-homologous.
Novikov's question was motivated by an application to the conductivity theory of monocrystals:
the investigated surface was a Fermi surface of a metal while the direction of the considered planes was orthogonal to the magnetic field.

Novikov's problem admits a natural formulation in terms of singular measured foliations on surfaces.
The foliations are defined by the restriction of
a differential 1-form on~$\mathbb T^3$ with constant coefficients to a null-homologous surface.
In 1984, A.\,Zorich~\cite{Zo} made the first step in the study of this problem
dealing with the case of a magnetic field close to rational.

Few years later, a more detailed study was undertaken by I.\,Dynnikov~\cite{D01,D02,D1}
(one interesting phenomenon was also discovered by~S.\,Tsarev),
which resulted in a qualitative classification of foliations arising
in Novikov's problem.
The only case that remained not fully understood dealt with so called chaotic regimes,
which were characterized by the presence of a minimal component of the foliation having
genus larger than one (in which case it is at least three).
As was discovered later, the Arnoux--Yoccoz foliation is isomorphic to one of the foliations
that appear as a chaotic regime in Novikov's problem.

The method described in~\cite{D1,D08} for constructing chaotic regimes in Novikov's problem
is based on a relation between two-dimensional CW-complexes bearing a measured
foliation and singular measured foliations on surfaces. It is used to produce
a surface with a singular measured foliation on it from a system of isometries
endowed with a little more combinatorial data so that the dynamical properties
of the foliation and those of the system of isometries are related in a certain way (which
is not fully understood in general). The construction has been given the name
`double suspension surface' in~\cite{DS2016} and used in~\cite{DS1,DS} to construct
exotic chaotic regimes in Novikov's problem.

In 1988, P.\,Arnoux~\cite{A} provided an explicit construction of a suspension surface associated with the
Arnouz--Yoccoz IET described in \cite{AY}. It is worth noting that his construction is quite different from the classical one (which is based on by zippered rectangles).

In 1991, the Arnouz--Yoccoz IET was generalized by P.\,Arnoux and G.\,Rauzy~\cite{AR}.  When working on a question in symbolic dynamics (namely, on a natural generalization of Sturmian words),  they described a family of IETs acting on a circle divided into six intervals. In this paper, we denote the surfaces endowed with a singular
measured foliation corresponding to the Arnoux--Rauzy IETs by~$(\Sigma_\lambda^{\mathrm{AR}},\mathscr F_\lambda^{\mathrm{AR}})$, where~$\lambda$ runs over the parameter space.

This class of IETs is identified by a special combinatorics, as well as a very particular choice of the lengths of the intervals: first, only three of them are rationally independent, but even in this three-dimensional parameter space
(two-dimensional after a normalization), Arnoux--Rauzy IET are given by a very special subset. This subset,
denoted here by~$\rg$, has a fractal structure and has been named \emph{the Rauzy gasket} in~\cite{AS}.

The Rauzy gasket has also appeared in mathematics quite a few times. In 1993, G.\,Levitt studied pseudogroups of rotations on a circle and described the Rauzy gasket as a subset of the parameter space that gave rise to minimal dynamical systems of this class~\cite{L}. Later, the same fractal was rediscovered by I.\,Dynnikov and R.\,De Leo in connection with Novikov's problem~\cite{DD}. They considered singular measured foliations on a specific PL-surface (so called the regular skew polyhedron~$\{4,6\,|\,4\}$), for which
we use the notation~$(\Sigma^{\mathrm{PL}},\mathscr F_\lambda^{\mathrm{PL}})$ in this paper.
Also, in \cite{AS} the Rauzy gasket was studied from the point of view of symbolic dynamics (Arnoux--Rauzy words etc) and multidimensional fraction algorithms (the fully subtractive algorithm and a dual one).

In all these papers it is proved that the Rauzy gasket has zero Lebesgue measure, and all these proofs are morally very different. Namely, the proof in \cite{AS} relies on a classical argument from ergodic theory and is in fact due to \cite{MeNo}, the proof in~\cite{DD} uses some tools from measure theory while the proof in \cite{L} is attributed to J.-C.\,Yoccoz and is very original in its spirit. 

In \cite{AHS1}, the Rauzy gasket is defined as a subset of the parameter space for certain systems of isometries of order three. These systems are denoted~$S_\lambda$ here,
where~$\lambda$ runs over the parameter space. The points of~$\rg$ are singled out by
the condition that the corresponding systems of isometries are of thin (exotic) type. It is also shown there that the Hausdorff dimension of this set is strictly less than two.

In~\cite{AHS2}, an invariant ergodic
probability measure (denoted by~$\mu_\rg$ in the present paper)  on~$\mathscr R$ consistent with the normalization process that is used to define~$\rg$
has been constructed. A particular case of the construction from~\cite{D08}
relating singular measured foliations on surfaces with systems of isometries
and the technique of O.\,Sarig~\cite{S0,S2}
are used in~\cite{AHS2} to study dynamical properties of trajectories
in a special case of Novikov's problem.

In the current paper we outline all the constructions mentioned above and explain why all of them lead to the same family of foliations. 

We start in Section~\ref{RG} with the definition of the Rauzy gasket. In Section~\ref{AR-sec}, we define the Arnoux--Rauzy IETs and the corresponding suspension
surfaces endowed with a singular measured foliation~$(\Sigma_\lambda^{\mathrm{AR}},\mathscr F_\lambda^{\mathrm{AR}})$.  The relation to symbolic dynamical systems
(namely, Arnoux--Rauzy words) is indicated in Section~\ref{words-seq}. In Section~\ref{foliation-sec}, we explain how to obtain the singular measured foliations associated with the Arnoux--Rauzy IETs as a cover~$(\Sigma_\lambda,\mathscr F_\lambda)$ of singular measured foliations on the projective plane ${\mathbb R}P^2$ (this is a natural generalization of the Arnoux--Yoccoz construction). In Section~\ref{Novikovproblem}, we describe the construction of chaotic foliations~$\mathscr F_\lambda^{\mathrm{PL}}$
on a piecewise linear surface~$\Sigma^{\mathrm{PL}}$ in Novikov's problem
which are introduced in \cite{DD}.
Finally, in Sections~\ref{systemsofisometries} and \ref{double-susp-sec},
we describe a family~$(S_\lambda,\theta_\lambda)$ of enhanced systems of isometries
and the double suspension surface construction for them, which produces
yet another family of surfaces endowed with a singular measured foliation.

Propositions~\ref{ar=rp2-prop}, \ref{pl=ar-prop}, and~\ref{double=ar-prop}
stated and proven below can be summarized as follows.
\begin{theorem}\label{equivalence}
For any~$\lambda\in\Delta\setminus\partial\Delta$, the foliations~$\mathscr F_\lambda^{\mathrm{AR}}$,
$\mathscr F_\lambda$, $\mathscr F_\lambda^{\mathrm{PL}}$, as well as the one
obtained by means of the double suspension surface construction from~$(S_\lambda,\theta_\lambda)$,
are all equivalent.
\end{theorem}
\begin{remark}
Recently, P.\,Hubert and O.\,Paris-Romaskevich found out that the Rauzy gasket also emerges in connection with tiling billiards~\cite{HPR}. The detailed description of the connection between tiling billiards and Novikov's problem will be explained in the forthcoming paper \cite{DHMPRS}.

\end{remark} 
Using this equivalence, we also refine the ergodic properties of the foliations from the discussed family. Namely, 
Corollaries~\ref{main-coro} and~\ref{two-measure-cor} below can be summarized as follows.

\begin{theorem}\label{uniqueergodicity} 
Almost all \emph(with respect to the measure $\mu_\rg$ mentioned above\emph) Arnoux--Rauzy IETs are uniquely ergodic, and non of them admits more than two invariant ergodic measures. 
\end{theorem}

The existence of Arnoux--Rauzy IETs with two different invariant probability
measures is established in \cite{DS}.
On the other hand, it is known (see \cite{B}) that \emph{all} Arnoux--Rauzy words are uniquely ergodic. There is no contradiction with our Theorem \ref{equivalence}, since
the symbolic dynamical system associated with an Arnoux--Rauzy word is in general only a factor of the respective Arnoux--Rauzy IET. A one-to-one correspondence between an Arnoux--Rauzy IET and its combinatorial model is proved in \cite{ACFH} under a diophantine condition. This result provides a strong version of Theorem \ref{uniqueergodicity} with a very different proof based on combinatorial arguments. 

\section*{Acknowledgements} The third author was partially supported by RFBR-CNRS grant No. 18-51--15010.

\section{The Rauzy gasket}\label{RG}
Denote by~$\Delta$ the standard two-simplex, and by~$\lambda_1,\lambda_2,\lambda_3$ the barycentric coordinates on it:
$$\Delta = \{(\lambda_1, \lambda_2, \lambda_3)\in \mathbb{R}^3_{\geqslant 0}: \lambda_1+ \lambda_2 +\lambda_3 =1\}.$$
Let~$L_1$, $L_2$, $L_3$ be the linear maps~$\mathbb R^3_{\geqslant0}\rightarrow\mathbb R^3_{\geqslant0}$
defined by the matrices
$$M_1=\begin{pmatrix}1&1&1\\0&1&0\\0&0&1\end{pmatrix},\quad
M_2=\begin{pmatrix}1&0&0\\1&1&1\\0&0&1\end{pmatrix},\quad
M_3=\begin{pmatrix}1&0&0\\0&1&0\\1&1&1\end{pmatrix},$$
respectively, and let~$f_1$, $f_2$, $f_3$ be the respective projective maps from~$\Delta$ to itself:
$$\begin{aligned}
f_1(\lambda_1,\lambda_2,\lambda_3)&=\Bigl(\frac1{1+\lambda_2+\lambda_3},\frac{\lambda_2}{1+\lambda_2+\lambda_3},\frac{\lambda_3}{1+\lambda_2+\lambda_3}\Bigr),\\
f_2(\lambda_1,\lambda_2,\lambda_3)&=\Bigl(\frac{\lambda_1}{1+\lambda_1+\lambda_3},\frac1{1+\lambda_1+\lambda_3},\frac{\lambda_3}{1+\lambda_1+\lambda_3}\Bigr),\\
f_3(\lambda_1,\lambda_2,\lambda_3)&=\Bigl(\frac{\lambda_1}{1+\lambda_1+\lambda_2},\frac{\lambda_2}{1+\lambda_1+\lambda_2},\frac1{1+\lambda_1+\lambda_2}\Bigr).
\end{aligned}$$

The images~$f_1(\Delta)$, $f_2(\Delta)$, $f_3(\Delta)$ are the three smaller triangles cut off from~$\Delta$ by the midlines (see Fig.~\ref{midlines-fig}).
\begin{figure}[ht]
\includegraphics[scale=.7]{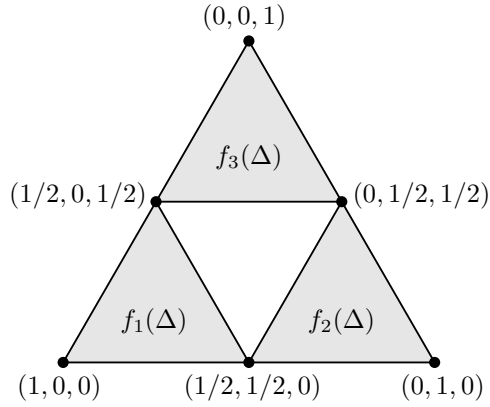}\put(-165,-5){$(1,0,0)$}\put(-102,-5){$(1/2,1/2,0)$}\put(-20,-5){$(0,1,0)$}
\put(-168,68){$(1/2,0,1/2)$}\put(-38,68){$(0,1/2,1/2)$}\put(-95,135){$(0,0,1)$}
\put(-126,20){$f_1(\Delta)$}\put(-55,20){$f_2(\Delta)$}\put(-90,82){$f_3(\Delta)$}
\caption{The images of~$\Delta$ under~$f_i$, $i=1,2,3$}\label{midlines-fig}
\end{figure}

For a word~$\slbf i=i_1i_2\ldots i_k$ in the alphabet~$\{1,2,3\}$
we denote by~$f_{\slbf i}$
the composition~$f_{i_1}\circ\ldots\circ f_{i_k}$. In particular, $f_\varnothing=\mathrm{id}$.

\begin{definition}
\emph{The Rauzy gasket} is the maximal subset~$\rg$ of~$\Delta$ satisfying the following:
$$\rg=f_1(\rg)\cup f_2(\rg)\cup f_3(\rg).$$
\end{definition}

In other words, $\lambda\in\rg$ if and only if there is an infinite word~$i_1i_2i_3\ldots\in\{1,2,3\}^{\mathbb N}$
such that, for any~$k\in\mathbb N$, we have~$\lambda\in f_{i_1\ldots i_k}(\Delta)$.
Such an~$\slbf i$ will be called \emph{a directing word for~$\lambda$}.

Denote by~$\Delta_0$ the subset of~$\Delta$ consisting of all points satisfying the triangle inequalities:
$$\Delta_0=\{(\lambda_1,\lambda_2,\lambda_3)\in\Delta:\lambda_1,\lambda_2,\lambda_3\leqslant1/2\}.$$
One can see that the union~$f_1(\Delta)\cup f_2(\Delta)\cup f_3(\Delta)$ is
the closure of~$\Delta\setminus\Delta_0$.

The Rauzy gasket can also be described as the subset of~$\Delta$ obtained by removing
the interior of all the triangles of the form~$f_{\slbf i}(\Delta_0)$, where~$\slbf i$
runs over the set of all finite words in the alphabet~$\{1,2,3\}$. Topologically,
the Rauzy gasket is equivalent to the Sierpinski triangle, but the geometry
is quite different, see Figure~\ref{RaGu}.
\begin{figure}[ht]
\includegraphics[scale=1.1]{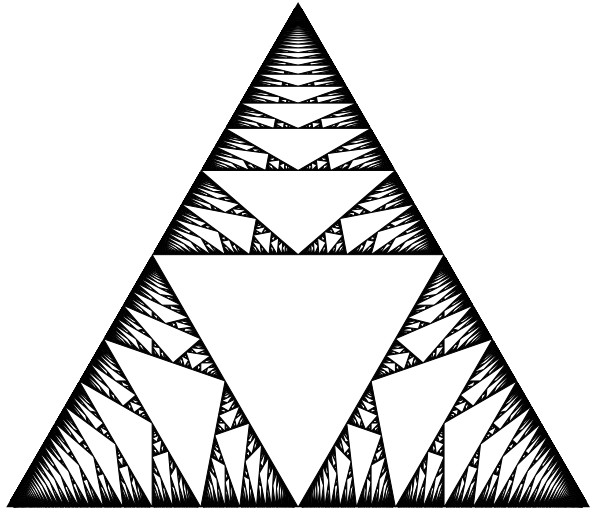}
\caption{The Rauzy Gasket}
\label{RaGu}
\end{figure}

We will be particularly interested in totally irrational points of~$\rg$, that is, points~$(\lambda_1,\lambda_2,\lambda_3)\in\rg$ such that~$\lambda_1$,
$\lambda_2$, $\lambda_3$ are linearly independent over~$\mathbb Q$. The subset of~$\rg$ consisting
of such points will be denoted by~$\rg_{\mathrm{irr}}$.

\begin{proposition}\label{regular-rauzy-prop}
\emph{(i)}
We have
\begin{equation}\label{regular-rauzy-eq}
\rg\setminus\rg_{\mathrm{irr}}=\bigcup_{\slbf i}f_{\slbf i}(\partial\Delta),\end{equation}
where the union is taken over all finite words~$\slbf i$ in the alphabet~$\{1,2,3\}$.

\emph{(ii)}
The correspondence~$\lambda\mapsto(\text{a directing word for }\lambda)$ restricted
to~$\rg_{\mathrm{irr}}$ is a homeomorphism
from~$\rg_{\mathrm{irr}}$ to the set~$\mathscr X$ of all infinite words in the alphabet~$\{1,2,3\}$
in which every letter~$i\in\{1,2,3\}$ appears infinitely often
endowed with the topology induced from the product topology on~$\{1,2,3\}^{\mathbb N}$.
\end{proposition}

\begin{remark}
Part~(i) of this proposition follows from the more general statement \cite[Lemma 4.2]{BSTh} proof of which relies
on a quite advanced symbolic dynamics. We provide here a straightforward proof.
To prove part~(ii), we use~(i) and follow Arnoux--Rauzy~\cite{AR}.
\end{remark} 

\begin{proof}[Proof of Proposition~\ref{regular-rauzy-prop}]
(i)
Denote the union in the right hand side of~\eqref{regular-rauzy-eq} by~$\boundaries$.
Clearly, $f_{\slbf i}(\partial\Delta)$ is disjoint from~$\rg_{\mathrm{irr}}$ for any~$\slbf i$.
We also have~$\partial\Delta\subset f_1(\partial\Delta)\cup f_2(\partial\Delta)\cup f_3(\partial\Delta)$,
hence~$\partial\Delta\subset\rg$ and~$f_{\slbf i}(\partial\Delta)\subset\rg$ for any~$\slbf i$. This implies
\begin{equation}\label{incl1}
\boundaries\subset\rg\setminus\rg_{\mathrm{irr}}.\end{equation}

\def\badset{\mathscr B}
Denote by $\badset$ the set of~$\lambda\in\rg$ such that not all the elements of~$\{1,2,3\}$
appear infinitely often in some directing word for~$\lambda$.
We claim that
\begin{equation}\label{incl2}
\badset\subset\boundaries.\end{equation}

To see this, consider the inverse maps of~$f_i$, $i=1,2,3$:
$$\begin{aligned}
f_1^{-1}(\lambda_1,\lambda_2,\lambda_3)&=\bigl((\lambda_1-\lambda_2-\lambda_3)/\lambda_1,
\lambda_2/\lambda_1,\lambda_3/\lambda_1\bigr),\\
f_2^{-1}(\lambda_1,\lambda_2,\lambda_3)&=\bigl(\lambda_1/\lambda_2,(\lambda_2-\lambda_1-\lambda_3)/\lambda_2,\lambda_3/\lambda_2\bigr),\\
f_3^{-1}(\lambda_1,\lambda_2,\lambda_3)&=\bigl(\lambda_1/\lambda_3,\lambda_2/\lambda_3,(\lambda_3-\lambda_1-\lambda_2)/\lambda_3\bigr).
\end{aligned}$$

A directing word~$i_1i_2\ldots$ for~$\lambda$ is defined by demanding that
the maps~$f_{i_1}^{-1}$, $f_{i_2}^{-1}$, $\ldots$ can be successively applied to~$\lambda$,
and, at every step, the result remains in~$\Delta$.

Suppose that~$\lambda_1>0$. Then, for~$j\in\{2,3\}$, the first coordinate of~$f_j^{-1}(\lambda)$
is not smaller than~$\lambda_1/(1-\lambda_1)>\lambda_1$. Note that any sequence $(x_0,x_1,x_2,\ldots)$
satisfying~$x_{i+1}=x_i/(1-x_i)$ for all~$i$, tends to infinity, when~$i$ goes to infinity,
if~$x_0>0$. Therefore, if~$\lambda_1>0$,
one cannot successively apply maps in~$\{f_2^{-1},f_3^{-1}\}$ to~$\lambda$ infinitely long without
making~$\lambda$
escape from~$\Delta$. Due to symmetry the same holds after any permutation of the indices~$1,2,3$.

Thus, either one of the coordinates of~$f_{i_1i_2\ldots i_k}^{-1}(\lambda)$, $k=1,2,3,\ldots$,
eventually vanishes, which means that~$\lambda_1,\lambda_2,\lambda_3$ are linearly dependent
over~$\mathbb Q$, or all elements of~$\{1,2,3\}$ appear in~$i_1i_2\ldots$
infinitely often. This implies~\eqref{incl2}.

Now we will show that
\begin{equation}\label{incl3}
\rg\setminus\rg_{\mathrm{irr}}\subset\badset,\end{equation}
which, together with~\eqref{incl1} and~\eqref{incl2}, implies~\eqref{regular-rauzy-eq}.

Let~$\lambda\in\rg$ and~$n=(n_1,n_2,n_3)\in\mathbb Z^3\setminus\{0\}$ be such that~$\langle n,\lambda\rangle=0$,
where~$\langle n,\lambda\rangle$ stands for~$n_1\lambda_1+n_2\lambda_2+n_3\lambda_3$,
and let~$i_1i_2\ldots$ be a directing word for~$\lambda$. Denote by~$\widehat f_i$, $i=1,2,3$,
the following linear isomorphisms of~$\mathbb Z^3$:
$$\begin{aligned}
\widehat f_1(m_1,m_2,m_3)&=(m_1,m_2+m_1,m_3+m_1),\\
\widehat f_2(m_1,m_2,m_3)&=(m_1+m_2,m_2,m_3+m_2),\\
\widehat f_3(m_1,m_2,m_3)&=(m_1+m_3,m_2+m_3,m_3).
\end{aligned}$$
For any~$k=1,2,3,\ldots$, we have
$$\bigl\langle(\widehat f_{i_k}\circ\ldots\circ\widehat f_{i_2}\circ\widehat f_{i_1})(n),f_{i_1i_2\ldots i_k}^{-1}(\lambda)\bigr\rangle=0.$$
Denote $n^k=(n^k_1,n^k_2,n^k_3)=(\widehat f_{i_k}\circ\ldots\circ\widehat f_{i_2}\circ\widehat f_{i_1})(n)$
and~$\lambda^k=f_{i_1i_2\ldots i_k}^{-1}(\lambda)$.

Denote also by~$q$ the following quadratic form on~$\mathbb Z^3$:
$$q(m)={m_1}^2+{m_2}^2+{m_3}^2-2m_1m_2-2m_1m_3-2m_2m_3,\quad m=(m_1,m_2,m_3).$$
The value of this form at~$m$ can be negative only if all coordinates
of~$m$ are strictly positive or all are strictly negative.
We also have
\begin{equation}\label{quadratic-form-properties}
q\bigl(\widehat f_i(m)\bigr)-q(m)=-4{m_i}^2\leqslant0\quad\text{and}\quad q\bigl(\widehat f_i(m)\bigr)=q(m)\Leftrightarrow\widehat f_i(m)=m.
\end{equation}

For no~$k$, all three numbers~$n^k_1,n^k_2,n^k_3$ can be
distinct from zero and be of the same sign, since~$\langle n^k,\lambda_k\rangle=0$ and~$\lambda^k\in\Delta$.
Therefore, we have~$q(n^k)\geqslant0$ for all~$k$. In view of~\eqref{quadratic-form-properties} this means
that the sequence~$(n,n^1,n^2,n^3,\ldots)$ eventually stabilizes: for some~$k$, we have~$n^k=n^{k+1}=n^{k+2}=\ldots$.
The maps~$\widehat f_i$ are invertible, so~$n^k\ne0$. If~$n^k_j\ne0$, then~$j$ does not appear
in the infinite word~$i_{k+1}i_{k+2}i_{k+3}\ldots$.

We have shown that~$\lambda\in\rg\setminus\rg_{\mathrm{irr}}$ implies~$\lambda\in\mathscr B$
and thus established that~$\rg\setminus\rg_{\mathrm{irr}}=\mathscr B=\mathscr C$.

(ii)
We have also shown that a directing word for any~$\lambda\in\rg_{\mathrm{irr}}$ is unique.
We denote it by~$\chi(\lambda)$.

Thus, we have a well defined map~$\chi:\rg_{\mathrm{irr}}\rightarrow\mathscr X$. This map is continuous,
since, for any finite word~$\slbf i=i_1i_2\ldots i_k$, the preimage of~$\{\slbf a\in\mathscr X:\slbf i\text{ is a prefix of }
\slbf a\}$ under~$\chi$ is~$f_{\slbf i}(\Delta\setminus\partial\Delta)\cap\rg_{\mathrm{irr}}$,
which is an open subset of~$\rg_{\mathrm{irr}}$.

It remains to show that any infinite word~$\slbf i=i_1i_2i_3\ldots\in\mathscr X$
is a directing word for a unique~$\lambda\in\rg$, and that~$\chi^{-1}$ is continuous.
The existence of~$\lambda\in\chi^{-1}(\slbf i)$ follows from the inclusions
$$\Delta\supset f_{i_1}(\Delta)\supset f_{i_1i_2}(\Delta)\supset\ldots\supset f_{i_1i_2\ldots i_k}(\Delta)\supset\ldots$$
and the fact that all these subsets are closed, since one can
take for~$\lambda$ any of their common intersection points.

The uniqueness of~$\lambda$ can be seen as follows. It is an easy check that,
for any~$k\in\mathbb N$ and~$\{j_1,j_2,j_3\}=\{1,2,3\}$
the matrix~$M_{j_1}M_{j_2}^kM_{j_3}A^{-1}$, where
$$A=\begin{pmatrix}3&1&1\\1&3&1\\1&1&3\end{pmatrix},$$
has only non-negative entries. Therefore, the restriction of~$f_{j_1j_2^kj_3}$ to the interior of~$\Delta$
is a contraction map with respect to the Hilbert projective metric, and
the contraction factor is not larger than that of the projective map defined by the matrix~$A$.
The contraction factor of the latter is strictly smaller than~$1$, since all entries of~$A$ are strictly positive.
It remains to notice that factors of the form~$j_1j_2^kj_3$ with~$k\in\mathbb N$ and~$\{j_1,j_2,j_3\}=\{1,2,3\}$
appear in~$\slbf i$ infinitely often. Hence, the intersection
$$\bigcap_{k=1}^\infty f_{i_1i_2\ldots i_k}(\Delta)$$
is a single point.

The continuity of~$\chi^{-1}$ follows from the obvious fact that, the closer~$\lambda'\in\rg_{\mathrm{irr}}$
to $\lambda$ the larger is the maximal common prefix of~$\chi(\lambda)$ and~$\chi(\lambda')$.
\end{proof}

It is shown independently in~\cite{L}, \cite{DD}, and~\cite{AS} that the Lebesgue measure of~$\rg$ is zero (see also \cite{MeNo} for a similar statement
about the fully subtractive algorithm). The following stronger result is established in~\cite{AHS1} (upper bound) and \cite{GRM} (lower bound):

\begin{theorem}
The Hausdorff dimension of~$\rg$ is strictly between 1.19 and 2.
\end{theorem}

Recently, Ch.\,Fougeron announced a
proof that the Hausdorff dimension of~$\rg$ is bounded from above by $1.825$ \cite{Fo20}.
The exact value of the Hausdorff dimension of~$\rg$ is unknown. A numerical
estimation for the dimension is~$\approx1.7$~\cite{DD}.

We now define the measure on~$\rg$ that has been introduced in~\cite{AHS2}
using the thermodynamic formalism for Markov shifts~\cite{S0,S2}, as well as the
classical suspension construction suggested in~\cite{V} and some important estimations from \cite{AGY}.

Denote by~$F$ the following three-to-one map on~$\rg_{\mathrm{irr}}$:
\begin{equation}\label{mapF-eq}
F(\lambda)=(f_i^{-1}(\lambda)),\quad\text{if }\lambda\in f_i(\rg_{\mathrm{irr}}).
\end{equation}
Define also \emph{the roof function}~$r:\rg_{\mathrm{irr}}\rightarrow\mathbb R$ by
$$r(\lambda)=-\log\max_i\lambda_i,$$
and let~$X$ be the following space:
$$X=(\rg_{\mathrm{irr}}\times\mathbb R)/\bigl\{(\lambda,a+r(\lambda))\sim(F(\lambda),a)\bigr\}.$$

The next proposition follows from~\cite{S0} and \cite{S2} and was established in \cite{AHS2}.

\begin{proposition}\label{max-entropy-prop}
There is a unique probability measure~$\mu_X$ on~$X$ that maximises the Kolmogorov entropy of the semiflow~$\phi$ on~$X$
defined by
$$\phi_t(\lambda,a)=(\lambda,a+t),\quad t\geqslant0.$$
\end{proposition}

\begin{definition}
The division of the measure~$\mu_X$ from Propositoin~\ref{max-entropy-prop} by
the Lebesgue measure~$dt$ on~$\mathbb R$
gives rise to a measure on~$\rg_{\mathrm{irr}}$, which will
be denoted by~$\mu_{\rg}$ and called \emph{the natural measure on~$\rg$}.
It is continued to the whole of~$\rg$ by putting~$\mu_{\rg}(\rg\setminus\rg_{\mathrm{irr}})=0$.
\end{definition}

The natural measure on~$\mathscr R$ has the following properties
which are established in~\cite{AHS2} and follow from O.\,Sarig's works~\cite{S0,S2}.

\begin{theorem}\label{measure-properties-th}
\emph{(i)} The total measure of the Rauzy gasket with respect to~$\mu_{\rg}$ is finite.
\\
\emph{(ii)} The map~$F$ defined by~\emph{\eqref{mapF-eq}} is ergodic with respect to~$\mu_{\rg}$.
\\
\emph{(iii)} For any open subset~$X\subset\rg$, we have~$\mu_{\rg}(X)>0$.
\end{theorem}

\section{Arnoux--Rauzy IETs} \label{AR-sec}
The first construction that naturally gave rise to the study of the Rauzy gasket appeared in \cite{AR}
in connection with a family of IETs which we now describe.

Denote by~$\mathbb S^1$ the circle~$\mathbb R/\mathbb Z$.

\begin{definition}
For~$\lambda\in\Delta$, we define \emph{the Arnoux--Rauzy IET}~$T_\lambda^{\mathrm{AR}}$
as the composition~$\Phi\circ T_\lambda':\mathbb S^1\rightarrow\mathbb S^1$, where~$\Phi$
is the translation (rotation) by~$1/2$,
that is,~$\Phi(x)=x+1/2$, and~$T_\lambda'$ is the IET of the interval~$[0,1)$ defined by the permutation
$$\begin{pmatrix}1&2&3&4&5&6\\2&1&4&3&6&5\end{pmatrix}$$
and parameters~$\bigl(\frac{\lambda_1}2,\frac{\lambda_1}2,\frac{\lambda_2}2,\frac{\lambda_2}2,\frac{\lambda_3}2,\frac{\lambda_3}2\bigr)$.
The explicit formula is this:
$$T^{\mathrm{AR}}_\lambda(x)=\left\{
\begin{aligned}
&x+\frac{1+\lambda_1}2,&&\text{if }x\in[0,\lambda_1/2)+\mathbb Z,\\
&x+\frac{1-\lambda_1}2,&&\text{if }x\in[\lambda_1/2,\lambda_1)+\mathbb Z,\\
&x+\frac{1+\lambda_2}2,&&\text{if }x\in[\lambda_1,\lambda_1+\lambda_2/2)+\mathbb Z,\\
&x+\frac{1-\lambda_2}2,&&\text{if }x\in[\lambda_1+\lambda_2/2,\lambda_1+\lambda_2)+\mathbb Z,\\
&x+\frac{1+\lambda_3}2,&&\text{if }x\in[\lambda_1+\lambda_2,\lambda_1+\lambda_2+\lambda_3/2)+\mathbb Z,\\
&x+\frac{1-\lambda_3}2,&&\text{if }x\in[\lambda_1+\lambda_2+\lambda_3/2,1)+\mathbb Z.\\
\end{aligned}
\right.$$

We also define
$$\widetilde T^{\mathrm{AR}}_\lambda(x)=\lim_{y\rightarrow x-0}T^{\mathrm{AR}}_\lambda(y).$$
\end{definition}

\begin{definition}
\emph{An orbit} of the transformation~$T^{\mathrm{AR}}_\lambda$ (respectively,~$\widetilde T^{\mathrm{AR}}_\lambda$) refers
to any minimal subset~$X\in\mathbb S^1$ such that~$T^{\mathrm{AR}}_\lambda(X)=X$ (respectively,~$\widetilde T^{\mathrm{AR}}_\lambda(X)=X$).

An orbit is called \emph{singular} if it contains a point where the transformation is discontinuous, that is,
a point from~$\{0,\lambda_1/2,\lambda_1,\lambda_1+\lambda_2/2,\lambda_1+\lambda_2,1-\lambda_3/2\}(\mathrm{mod}\,\mathbb Z)$, and \emph{regular} otherwise.
\end{definition}

Regular orbits of~$T^{\mathrm{AR}}_\lambda$ and~$\widetilde T^{\mathrm{AR}}_\lambda$ are clearly the same,
but singular orbits are different.

\begin{proposition}\label{AR-minimal-prop}
For~$\lambda\in\Delta$, the Arnoux--Rauzy IET~$T^{\mathrm{AR}}_\lambda$ is minimal if and only if~$\lambda\in\rg_{\mathrm{irr}}$.
The transformation~$T^{\mathrm{AR}}_\lambda$ \emph(or~$\widetilde T^{\mathrm{AR}}_\lambda$\emph) has a regular finite orbit if and only if~$\lambda\not\in\rg$
or~$\lambda\in\mathbb Q^3$.
\end{proposition}

\begin{proof}
We follow the ideas of~\cite{AR}.

Clearly, if~$\lambda\in\Delta\cap\mathbb Q^3$ then all orbits of~$T^{\mathrm{AR}}_\lambda$ are finite. We assume~$\lambda\notin\mathbb Q^3$
from now till the end of the proof.

Suppose that~$\lambda_1$, $\lambda_2$, $\lambda_3$ satisfy the strict triangle inequalities,
that is, $\lambda\in\Delta_0\setminus\partial\Delta_0$. Then it is a direct check that~$T^{\mathrm{AR}}_\lambda$ has
regular orbits of order three. For instance, if~$1/2>\lambda_1\geqslant\lambda_2\geqslant\lambda_3$,
then~$\bigl(T^{\mathrm{AR}}_\lambda\bigr)^3(x)=x$ for all~$x\in\bigl((\lambda_1-\lambda_3)/2,\lambda_2/2\bigr)$.
Other cases are similar.

If~$\lambda\in\partial\Delta$, then one can see that~$T^{\mathrm{AR}}_\lambda$ is not minimal and has no
finite orbits. For instance, if~$\lambda_3=0$,
then the interval~$\bigl[\lambda_1/2,(\lambda_1+1)/2\bigr)$ and its complement in~$\mathbb S^1$ are invariant under~$T^{\mathrm{AR}}_\lambda$,
and the restriction of~$T^{\mathrm{AR}}_\lambda$ to each of these invariant subsets is conjugate to an irrational rotation,
hence all orbits are infinite.

To treat the general case, we apply a version of the Rauzy induction.
The point is that, for~$\lambda\in f_i(\Delta)$, $i=1,2,3$,
the transformations~$T^{\mathrm{AR}}_\lambda$ and~$T^{\mathrm{AR}}_{f_i^{-1}(\lambda)}$ have the same
dynamical properties. Indeed, let~$\lambda=f_1(\lambda')$, $\lambda'\in\Delta$. This means that~$\lambda_1\geqslant1/2$.
Every orbit visits~$[0,\lambda_1)$, since $x\in[\lambda_1,1)$ implies~$T^{\mathrm{AR}}_\lambda(x)\in[0,\lambda_1)$.
Define~$T:[0,\lambda_1)\rightarrow[0,\lambda_1)$ by
$$T(x)=\left\{\begin{aligned}
&T^{\mathrm{AR}}_\lambda(x),&&\text{if }T^{\mathrm{AR}}_\lambda(x)\in[0,\lambda_1),\\
&T^{\mathrm{AR}}_\lambda\bigl(T^{\mathrm{AR}}_\lambda(x)\bigr),&&\text{otherwise.}
\end{aligned}\right.$$
By identifying~$\lambda_1$ with~$0$ we may view~$T$ as a transformation of the circle~$\mathbb R/(\lambda_1\mathbb Z)$.
It is then a direct check that~$T$ is conjugate to~$T^{\mathrm{AR}}_{\lambda'}$, and the conjugating map
is just a stretching with coefficient~$1/\lambda_1$.

If~$i=2$ or~$3$, we replace~$[0,\lambda_1)$ by~$[\lambda_1,\lambda_1+\lambda_2)$ or~$[\lambda_1+\lambda_2,1)$,
respectively, in the construction of~$T$.

If~$\lambda\notin\rg$, then, for some finite sequence~$\slbf i$ with entries in~$\{1,2,3\}$,
we have~$\lambda\in f_{\slbf i}(\Delta_0\setminus\partial\Delta_0)$,
and hence~$T^{\mathrm{AR}}_\lambda$ has regular finite orbits.

If~$\lambda\in\rg\setminus\rg_{\mathrm{irr}}$, then, according to Proposition~\ref{regular-rauzy-prop},
there is a finite sequence~$\slbf i$ such that~$\lambda\in f_{\slbf i}(\partial\Delta)$,
and hence~$T^{\mathrm{AR}}_\lambda$ is not minimal and has no regular finite orbits.

To see that~$T_\lambda^{\mathrm{AR}}$ is minimal for all~$\lambda\in\rg_{\mathrm{irr}}$ we introduce
a measured foliation for which~$T_\lambda^{\mathrm{AR}}$ is the Poincar\'e map.
For~$\lambda\in\Delta$, we define~$\Sigma^{\mathrm{AR}}_\lambda$ as `the mapping torus' of~$T^{\mathrm{AR}}_\lambda$,
that is, a surface obtained from the cylinder~$\mathbb S^1\times[0,1]$ by identifying~$(x,1)$ with~$(T^{\mathrm{AR}}(x),0)$
for all~$x\in\mathbb S^1$ (see the top picture in Figure~\ref{ar-surface-fig}).
\begin{figure}[ht]
\includegraphics[scale=.7]{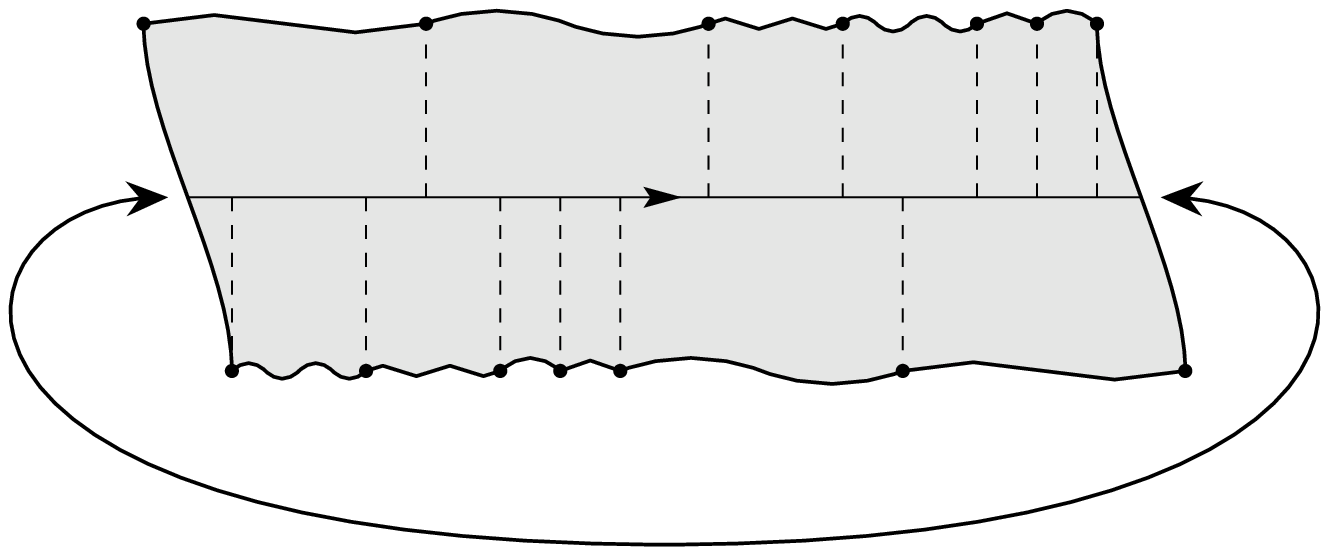}\put(-220,100){$1$}\put(-160,100){$2$}\put(-120,100){$3$}
\put(-93,100){$4$}\put(-74,100){$5$}\put(-62,100){$6$}
\put(-217,65){$4$}\put(-190,65){$3$}\put(-170,65){$6$}\put(-158,65){$5$}\put(-125,65){$2$}\put(-65,65){$1$}
\put(-155,20){identify}\put(-145,89){$\gamma_\lambda$}

\includegraphics[scale=.7]{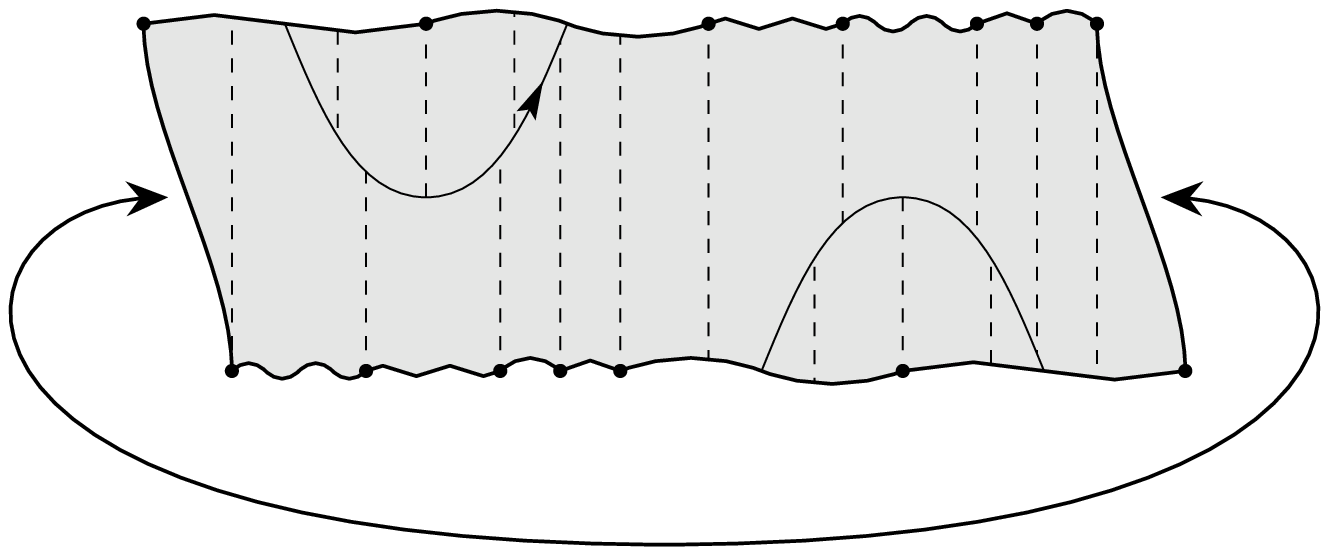}\put(-93,100){$4'$}\put(-74,100){$5'$}\put(-62,100){$6'$}
\put(-217,65){$4'$}\put(-190,65){$3'$}\put(-170,65){$6'$}\put(-158,65){$5'$}\put(-120,100){$3'$}
\put(-200,100){$1'$}\put(-182,100){$2'$}\put(-240,100){$1'$}
\put(-104,65){$2'$}\put(-86,65){$1'$}\put(-49,65){$1'$}
\put(-143,82){$2'$}\put(-213,110){$\scriptstyle6'$}\put(-169,110){$\scriptstyle3'$}
\put(-117,52){$\scriptstyle5'$}\put(-74,52){$\scriptstyle4'$}
\put(-155,20){identify}\put(-217,90){$\gamma_\lambda'$}
\caption{A step of the modified Rauzy induction}\label{rauzy-step-fig}\label{ar-surface-fig}
\end{figure}
The surface is endowed with a singular foliation denoted by~$\mathscr F^{\mathrm{AR}}_\lambda$,
coming from the foliation of the cylinder~$\mathbb S^1\times[0,1]$ by vertical arcs~$\{x\}\times[0,1]$.
One can see that, for~$\lambda\notin\partial\Delta$, the surface~$\Sigma^{\mathrm{AR}}_\lambda$
is an orientable compact surface of genus three, and the foliation~$\mathscr F^{\mathrm{AR}}_\lambda$
has two singularities, each of which is a monkey saddle. The Poincar\'e map of this foliation
on the transversal~$\gamma_\lambda=\mathbb S^1\times\{1/2\}$ (outside of six points where it
is not well defined) is precisely~$T^{\mathrm{AR}}_\lambda$. The decomposition into
six strips corresponding to the six intervals of the transformation is shown in dashed lines.

Minimality of~$T^{\mathrm{AR}}_\lambda$ is equivalent to that of~$\mathscr F_\lambda$, so
it suffices to prove the latter.

Suppose that~$\lambda\in f_i(\rg_{\mathrm{irr}})$.
As we
have seen above,
there is a subinterval of~$\gamma_\lambda$ on which the first return map of~$T_\lambda^{\mathrm{AR}}$,
is equivalent, after identifying the endpoints of the interval and a rescaling, to~$T_{f_i^{-1}(\lambda)}^{\mathrm{AR}}$.
This means that there is a closed transversal~$\gamma'_\lambda$
on which the Poincar\'e map induced by~$\mathscr F_\lambda$ is equivalent, after
rescaling, to~$T_{f_i^{-1}(\lambda)}^{\mathrm{AR}}$. Such a transversal
is shown at the bottom of Figure~\ref{rauzy-step-fig} in the case~$i=1$,
where the new decomposition into six strips is also indicated.

Thus, there is a homeomorphism~$\Sigma_\lambda\rightarrow\Sigma_{f^{-1}_i(\lambda)}$
that takes the foliation~$\mathscr F_\lambda$ to~$\mathscr F_{f_i^{-1}(\lambda)}$,
the transversal~$\gamma'_\lambda$ to~$\gamma_{f_i^{-1}(\lambda)}$, and
the Lebesgue measure on~$\gamma'_\lambda$ to the Lebesgue measure~$\gamma_{f_i^{-1}(\lambda)}$
multiplied by~$\lambda_i$.

Let~$i_1i_2i_3\ldots$ be the directing word for~$\lambda$.
Proceeding inductively, we can find a closed transversal~$\gamma_k$ of~$\mathscr F_\lambda$
on which the Poincar\'e map is equivalent, after rescaling, to~$T^{\mathrm{AR}}_{f_{i_1\ldots i_k}^{-1}(\lambda)}$.
The scaling coefficient equals
\begin{equation}\label{scaling-coef-eq}
\lambda_{i_1}\cdot\prod_{j=2}^{k}\bigl(f_{i_1\ldots i_{j-1}}^{-1}(\lambda)\bigr)_{i_j}.
\end{equation}
One can see that if~$i_j\ne i_{j+1}$, then~$\bigl(f_{i_1\ldots i_{j-1}}^{-1}(\lambda)\bigr)_{i_j}\leqslant3/4$.
Therefore, the transverse measure of~$\gamma_k$, which is equal to~\eqref{scaling-coef-eq},
tends to zero when~$k$ goes to infinity.

Suppose that~$\mathscr F_\lambda$ has a saddle connection~$\alpha$. Then, for some~$k$,
the transversal~$\gamma_k$ intersects~$\alpha$ exactly once.
This means that the image of some discontinuity point
of~$T^{\mathrm{AR}}_{\lambda'}$ is again a discontinuity point,
where~$\lambda'=f_{i_1\ldots i_k}^{-1}(\lambda)$.
It is a direct check that this does not happen when~$\lambda'\in\rg_{\mathrm{irr}}$, a contradiction.
Hence, $\mathscr F_\lambda$ has no saddle connections and thus is minimal.
\end{proof}

\begin{definition}\label{recurrent-def}
\emph{A factor} of an infinite word~$\slbf i=i_1i_2i_3\ldots$ is a subword of the form
$$i_ki_{k+1}\ldots i_{k+n-1},\quad k,n\in\mathbb N.$$

An infinite word~$\slbf i$ is called \emph{recurrent} if any factor of~$\slbf i$
appears in~$\slbf i$ infinitely often.
\end{definition}

For a measured foliation~$\mathscr F$ on a surface,
denote by~$\mathscr M(\mathscr F)$ the space of all invariant
transverse Borel measures of~$\mathscr F$\emph.

\begin{theorem}\label{typically-uniquely-ergodic-th}
For any~$\lambda\in\rg_{\mathrm{irr}}$ such that the directing word for~$\lambda$ is recurrent,
the IET~$T_\lambda^{\mathrm{AR}}$ is uniquely ergodic, or, equivalently,
$\dim\mathscr M(\mathscr F_\lambda)=1$.
\end{theorem}

\begin{proof}
By construction, there is a closed transversal~$\gamma_\lambda\subset\Sigma_\lambda=\Sigma_\lambda^{\mathrm{AR}}$
of the foliation~$\mathscr F_\lambda$ such that the Poincar\'e map induced by~$\mathscr F_\lambda$
on~$\gamma_\lambda$ coincides with~$T^{\mathrm{AR}}_\lambda$ after identifying~$\gamma_\lambda$
with the circle~$\mathbb S^1=\mathbb R/\mathbb Z$.

Suppose that~$\lambda\in f_i(\rg_{\mathrm{irr}})$.

Recall that the map~$T_\lambda^{\mathrm{AR}}$ is a composition of an IET
with permutation~$(12)(34)(56)$ and parameters~$(\lambda_1/2,\lambda_1/2,
\lambda_2/2,\lambda_2/2,\lambda_3/2,\lambda_3/2)$, and a rotation by~$1/2$.
Therefore, since~$T_\lambda^{\mathrm{AR}}$ is minimal,
any invariant measure~$\mu\in\mathscr M(\mathscr F_\lambda)$ is completely determined
by the following seven parameters:
$$\begin{aligned}
x_1&=\mu([0,\lambda_1/2]),\\
x_2&=\mu([\lambda_1/2,\lambda_1]),\\
x_3&=\mu([\lambda_1,\lambda_1+\lambda_2/2]),\\
x_4&=\mu([\lambda_1+\lambda_2/2,\lambda_1+\lambda_2]),\\
x_5&=\mu([\lambda_1+\lambda_2,\lambda_1+\lambda_2+\lambda_3/2],\\
x_6&=\mu([\lambda_1+\lambda_2+\lambda_3/2,1]),\\
x_7&=\mu([0,1/2])-x_1-x_3-x_5,
\end{aligned}$$
which satisfy the following conditions:
\begin{equation}\label{x-cond-eq}
x_1,x_2,x_3,x_4,x_5,x_6>0,\quad|x_7|<\sum_{j=1}^6x_j.
\end{equation}

We have seen in the proof of Proposition~\ref{AR-minimal-prop} that there is a
homeomorphism~$\varphi:\Sigma_\lambda\rightarrow\Sigma_{f^{-1}_i(\lambda)}$
that takes the foliation~$\mathscr F_\lambda$ to~$\mathscr F_{f_i^{-1}(\lambda)}$.

Let~$x_i'$, $i=1,\ldots,7$, be the seven parameters defined in the same
way for~$\varphi_*(\mu)$ viewed as an invariant measure for~$\mathscr F_{f_i^{-1}(\lambda)}$.
It is then a direct check that the vectors~$x=(x_1,\ldots,x_7)^\top$ and~$x'=(x_1',\ldots,x_7')^\top$ are related as follows:
$$x=\widetilde N_ix',$$
where
$$\widetilde N_1=\begin{pmatrix}
1&0&0&1&0&1&-1\\
0&1&1&0&1&0&1\\
0&0&1&0&0&0&0\\
0&0&0&1&0&0&0\\
0&0&0&0&1&0&0\\
0&0&0&0&0&1&0\\
0&0&0&0&0&0&1
\end{pmatrix},\quad
\widetilde N_2=\begin{pmatrix}
1&0&0&0&0&0&0\\
0&1&0&0&0&0&0\\
0&1&1&0&0&1&-1\\
1&0&0&1&1&0&1\\
0&0&0&0&1&0&0\\
0&0&0&0&0&1&0\\
0&0&0&0&0&0&1
\end{pmatrix},$$
$$\widetilde N_3=\begin{pmatrix}
1&0&0&0&0&0&0\\
0&1&0&0&0&0&0\\
0&0&1&0&0&0&0\\
0&0&0&1&0&0&0\\
0&1&0&1&1&0&-1\\
1&0&1&0&0&1&1\\
0&0&0&0&0&0&1
\end{pmatrix}.$$

Now let~$\slbf i=i_1i_2i_3\ldots$ be the directing word for~$\lambda$. By induction using
the reasoning above, one shows that
$$x=\widetilde N_{i_1}\widetilde N_{i_2}\ldots\widetilde N_{i_k}x^{(k)},$$
where the coordinates of the vector~$x^{(k)}\in\mathbb R^7$ also satisfy~\eqref{x-cond-eq}.
In particular, we have~$x_7=x_7^{(1)}=x_7^{(2)}=\ldots$, which implies~$|x_7|<\sum_{i=1}^6x_i^{(k)}$
for any~$k\in\mathbb N$. One can see that the sum~$\sum_{i=1}^6x_i^{(k)}$ tends to zero
as~$k$ goes to infinity, hence~$x_7=0$. So, the last coordinate of each vector~$x^{(k)}$
is zero, hence, it can be discarded.

For~$i=1,2,3$, let~$N_i$ be the matrix obtained from~$\widetilde N_i$ by deleting the last row
and the last column. Let also~$X_k$, where~$k\in\mathbb N$,
be the image of the positive cone~$\mathbb R_{>0}^6$
under the linear map with the matrix~$N_{i_1}N_{i_2}\ldots N_{i_k}$.

It is an easy check that all coordinates of the matrix~$N_1N_2N_3N_1N_2$ are strictly positive.
The same is true for any matrix of the form~$N_{j_1}N_{j_2}\ldots N_{j_m}$ provided
that
\begin{equation}\label{12312-eq}
\text{the sequence }j_1j_2\ldots j_m\text{ contains a subsequence $j_{k_1}j_{k_2}j_{k_3}j_{k_4}j_{k_5}=\text{`12312'}$}.
\end{equation}
The linear map defined by such a matrix on~$\mathbb R_{>0}^6$ is a compression
in the Hilbert projective metric with factor smaller than~$1$.

Since every letter from~$\{1,2,3\}$ appears in~$\slbf i$ infinitely often,
there exists a factor~$u=j_1j_2\ldots j_m$ of~$\slbf i$ satisfying~\eqref{12312-eq}.
Since~$\slbf i$ is recurrent, the factor~$u$ appears in~$\slbf i$ infinitely often,
which implies that the intersection~$\cap_{k=1}^\infty X_k$ is a single ray. The latter means
that the vector~$x$ is determined by the directing word~$\slbf i$ uniquely up to scale, which is equivalent to saying
that~$T_\lambda^{\mathrm{AR}}$ (equivalently, $\mathscr F_\lambda$) is uniquely ergodic.
\end{proof}

\begin{corollary}\label{main-coro}
For~$\mu_{\rg}$-almost all~$\lambda\in\rg$ the Arnoux--Rauzy IET~$T_\lambda^{\mathrm{AR}}$
is uniquely ergodic.
\end{corollary}

\begin{proof}
Recall (see Proposition~\ref{regular-rauzy-prop}) that we identified~$\rg_{\mathrm{irr}}$
with a subspace~$\mathscr X\subset\{1,2,3\}^{\mathbb N}$. Non-recurrent
words form a countable union of nowhere dense subsets of~$\mathscr X$.
This union is invariant under the map~$F$ defined by~\eqref{mapF-eq}.
Therefore, it has zero measure~$\mu_{\rg}$ due to Theorem~\ref{measure-properties-th}.
The claim now follows from Theorem~\ref{typically-uniquely-ergodic-th}.
\end{proof}

One can see from the proof of Theorem~\ref{typically-uniquely-ergodic-th} that the hypothesis of the
theorem can be considerably weakened. Namely, for~$T_\lambda^{\mathrm{AR}}$
to be uniquely ergodic, it suffices that \emph{some} finite word~$j_1j_2,\ldots,j_m$
in the alphabet~$\{1,2,3\}$ such that the matrix~$N_{j_1}N_{j_2}\ldots N_{j_m}$ has only strictly positive entries,
appears as a factor of~$\chi(\lambda)$ infinitely often.

The following statement is a slight generalization of a similar result of~\cite{DS}. It shows
that without any additional hypothesis on~$\lambda\in\rg_{\mathrm{irr}}$ the
assertion of Theorem~\ref{typically-uniquely-ergodic-th} fails.

\begin{proposition}\label{non-uniquely-ergodic-prop}
Suppose that the directing word for~$\lambda\in\rg_{\mathrm{irr}}$ has the form
$$\slbf i=i_1^{k_1}i_2^{k_2}i_3^{k_3}\ldots,$$
where~$i_j\ne i_{j+1}$, $k_j\in\mathbb N$ for all~$j=1,2,3,\ldots$, and the series~$\sum_{j=1}^\infty1/k_j$ converges.
Then~$T_\lambda^{\mathrm{AR}}$ is not uniquely ergodic.
\end{proposition}

\begin{proof}
Denote by~$\mathbbm1_n$ the identity $n\times n$-matrix, and by~$e_1,e_2,e_3$ the standard basis of~$\mathbb R^3$.
We use the construction from the proof of Theorem~\ref{typically-uniquely-ergodic-th}.
The matrices~$N_i$ satisfy the following relations:
\begin{equation}\label{Nik-eq}
N_i^k=k(N_i-\mathbbm1_6)+\mathbbm1_6=k\,P_i\otimes B+\mathbbm1_6,
\text{ where }P_i=M_i-\mathbbm1_3,\ B=\begin{pmatrix}0&1\\1&0\end{pmatrix},
\end{equation}
and~`$\otimes$' stands for the Kronecker product. We have
$$P_ie_j=\left\{\begin{aligned}&0&&\text{if }i=j,\\
&e_i&&\text{otherwise,}
\end{aligned}\right.$$
which implies the following:
\begin{equation}\label{PPP-eq}
P_{j_1}P_{j_2}\ldots P_{j_m}(e_1+e_2+e_3)=\left\{\begin{aligned}&2e_{j_1}&&\text{if }j_i\ne j_{i+1}\text{ for all }
i=1,\ldots,m-1,\\
&0&&\text{otherwise.}\end{aligned}\right.
\end{equation}

Denote by~$C(\ell,m)$, $\ell,m\in N$, the following $6\times2$-matrix:
$$C(\ell,m)=\frac{N_{i_\ell}^{k_\ell}}{k_\ell}\frac{N_{i_{\ell+1}}^{k_{\ell+1}}}{k_{\ell+1}}\ldots
\frac{N_{i_{\ell+m}}^{k_{\ell+m}}}{k_{\ell+m}}\bigl((e_1+e_2+e_3)\otimes B^{m+1}\bigr).
$$
From~\eqref{Nik-eq} we have:
$$C(\ell,m)=\prod_{j=\ell}^{\ell+m}\Bigl(P_{i_j}\otimes\mathbbm1_2+\frac{\mathbbm1_3\otimes B}{k_j}\Bigr)\bigl((e_1+e_2+e_3)\otimes\mathbbm1_2\bigr).$$

Since the series~$\sum_{j=1}^\infty1/k_j$ converges, one can see from~\eqref{PPP-eq}
that, for any~$\ell\in\mathbb N$, there exists a limit
$\lim_{m\rightarrow\infty}C(\ell,m)$,
which we denote by~$C(\ell,\infty)$. One can also see that
the difference
$$C(\ell,\infty)-2e_{i_\ell}\otimes\mathbbm1_2$$
tends to zero as~$\ell$ goes to infinity, so, the rank of~$C(\ell,\infty)$ equals two if~$\ell$ is large enough.
Therefore, $C(1,\infty)$ also has rank two because
$$C(1,\infty)=\frac{N_{i_1}^{k_1}}{k_1}\frac{N_{i_2}^{k_2}}{k_2}\ldots
\frac{N_{i_\ell}^{k_\ell}}{k_\ell}\,(\mathbbm1_3\otimes B^\ell)\,C(\ell+1,\infty),$$
and the matrices~$N_i$ are non-degenerate.
The two columns of~$C(1,\infty)$ represent two non-proportional
(and actually ergodic) invariant measures of~$T_\lambda^{\mathrm{AR}}$.
\end{proof}

Arnoux--Rauzy IETs have a mixing behavior which is
unusual for general IET. Namely, the following statement holds.

\begin{theorem}\label{weakmixing}
Arnoux--Rauzy IETs are almost never weakly mixing, with respect to the measure $\mu_{\rg}$. 
\end{theorem}

This statement follows directly from the series of results in \cite{BSTh} and \cite{ACFH}. Namely, Theorem 5 and Corollary 6 in \cite{BSTh} concern spectral properties of Arnoux--Rauzy words: it states that the Pisot property of the Lyapunov spectrum of Arnoux--Rauzy words implies that the corresponding dynamical system is almost never weakly mixing. Theorem 1 in \cite{ACFH} describes the precise connection between Arnoux--Rauzy words and Arnoux--Rauzy IETs. 

\begin{remark}
Proposition 21 in \cite{ACFH} implies that the statement of Theorem \ref{weakmixing} can not be strengthen
by removing~`almost': there exist explicit examples of weakly mixing Arnoux--Rauzy IETs. 
\end{remark}

\begin{remark}
The connection between two exceptional sets---the set of non-uniquely ergodic Arnoux--Rauzy IETs and the set of weakly mixing Arnoux--Rauzy trans\-for\-mation---is not well understood yet; it seems to be an interesting open problem. 
\end{remark}

P.\,Arnoux noticed in \cite{A} that the Arnoux--Yoccoz IET,
which is a particular case of a Arnoux--Rauzy IET,
is
semiconjugate to a translation on a 2-torus. Corollary~2 in \cite{ACFH} states that the property holds for almost all Arnoux--Rauzy IETs. 

\section{Arnoux--Rauzy words}\label{words-seq}
Arnoux--Rauzy words are defined in~\cite{AR} as a generalization of Sturmian words. They are intimately
related with the IETs~$T^{\mathrm{AR}}_\lambda$ for~$\lambda\in\rg_{\mathrm{irr}}$.
This relation is described below.

\begin{definition}
Let
$\slbf u=u_0u_1u_2u_3\ldots$
be an infinite word in some alphabet~$A$.
A factor~$w$ of~$\slbf u$ is called \emph{right special} (respectively, \emph{left special})
if there are two letters~$a,b\in A$, $a\ne b$, such that
both words~$wa$ and~$wb$ (respectively, $aw$ and~$bw$)
are factors of~$\slbf u$.
\end{definition}

\begin{definition}
\emph{An Arnoux--Rauzy word} is an infinite word~$u$
in the alphabet~$\{1,2,3\}$ such that the following holds:
\begin{enumerate}
\item
the word~$\slbf u$ is \emph{recurrent} (see Definition~\ref{recurrent-def});
\item
for any~$n\in\mathbb N$, the number of pairwise distinct factors of~$\slbf u$ of length~$n$
is equal to~$2n+1$;
\item
for any~$n\in\mathbb N$, there is exactly one right special factor
and exactly one left special factor of~$\slbf u$ of length~$n$.
\end{enumerate}
\end{definition}

With each pair~$(\lambda,x)$, where~$\lambda\in\Delta$ and~$z\in\mathbb S^1=\mathbb R/\mathbb Z$,
we associate an infinite word~$\slbf u(\lambda,x)$ in the alphabet~$\{1,2,3\}$ defined
as follows:
$$u_i(\lambda,x)=\left\{
\begin{aligned}
1,&\text{ if }\bigl(T^{\mathrm{AR}}_\lambda\bigr)^i(x)\in[0,\lambda_1);\\
2,&\text{ if }\bigl(T^{\mathrm{AR}}_\lambda\bigr)^i(x)\in[\lambda_1,\lambda_1+\lambda_2);\\
3,&\text{ if }\bigl(T^{\mathrm{AR}}_\lambda\bigr)^i(x)\in[\lambda_1+\lambda_2,1).
\end{aligned}
\right.$$
Similarly, we define a word~$\widetilde{\slbf u}(\lambda,x)$ in the alphabet~$\{1,2,3\}$ by
$$\widetilde u_i(\lambda,x)=\left\{
\begin{aligned}
1,&\text{ if }\bigl(\widetilde T^{\mathrm{AR}}_\lambda\bigr)^i(x)\in(0,\lambda_1];\\
2,&\text{ if }\bigl(\widetilde T^{\mathrm{AR}}_\lambda\bigr)^i(x)\in(\lambda_1,\lambda_1+\lambda_2];\\
3,&\text{ if }\bigl(\widetilde T^{\mathrm{AR}}_\lambda\bigr)^i(x)\in(\lambda_1+\lambda_2,1].
\end{aligned}
\right.$$
One can see that if the~$T^{\mathrm{AR}}_\lambda$-orbit of~$x$ is regular, then~$\slbf u(\lambda,x)=\widetilde{\slbf u}(\lambda,x)$.

With every infinite word~$\slbf u$ in an alphabet~$\mathscr A$ one associates a symbolic dynamical system~$(\Omega(\slbf u),\shift)$ in which~$\Omega(\slbf u)$
is the set of all infinite words~$\slbf w$ such that any factor of~$\slbf w$ is also a factor~$\slbf u$, and~$\shift:\Omega(\slbf u)\rightarrow
\Omega(\slbf u)$ is the shift operator:
$$\shift(u_1u_2u_3u_4\ldots)=u_2u_3u_4\ldots.$$
The set~$\Omega(\slbf u)$ is endowed with the topology induced from the product topology on~$\mathscr A^{\mathbb N}$.

\begin{theorem}\label{ar-words-th}
\emph{(i)} The
word~$\slbf u(\lambda,x)$ {\rm(}or~$\widetilde{\slbf u}(\lambda,x)${\rm)} is Arnoux--Rauzy if and only if~$\lambda\in\rg_{\mathrm{irr}}$.

\emph{(ii)} Every Arnoux--Rauzy word~$\slbf u$ has the form~$\slbf u(\lambda,x)$ or~$\widetilde{\slbf u}(\lambda,x)$ for some~$\lambda\in\rg_{\mathrm{irr}}$,
$x\in\mathbb S^1$. Moreover~$\lambda_1$, $\lambda_2$, and~$\lambda_3$ are the frequencies of the respective
letters in~$\slbf u$.

\emph{(iii)}
If~$\lambda\in\rg_{\mathrm{irr}}$, then~$\Omega(\slbf u(\lambda,x))=\Omega(\widetilde{\slbf u}(\lambda,x))=\Omega(\slbf u(\lambda,y))$
for any~$x,y\in\mathbb S^1$.

\emph{(iv)}
If~$\lambda\in\rg_{\mathrm{irr}}$, then the system~$(\Omega(\slbf u(\lambda,x)),\shift)$ is minimal
and uniquely ergodic.
\end{theorem}
This theorem follows from the results of~\cite{AR,AS,B} and Proposition~\ref{regular-rauzy-prop}.

\section {Foliations on $\mathbb R P^2$}\label{foliation-sec}
For each~$\lambda\in\Delta$, we define a measured foliation with singularities, which we denote by~$\mathscr F_\lambda^0$,
on the projective plane~$\mathbb RP^2$, as shown in Figure~\ref{foli-fig} (after cutting along a closed transversal).
Opposite points of the boundary
of the shown disc should be identified to obtain the original surface. The transverse measure of the boundary of the shown disc
is~$\lambda_1+\lambda_2+\lambda_3=1$.
\begin{figure}[ht]
\includegraphics{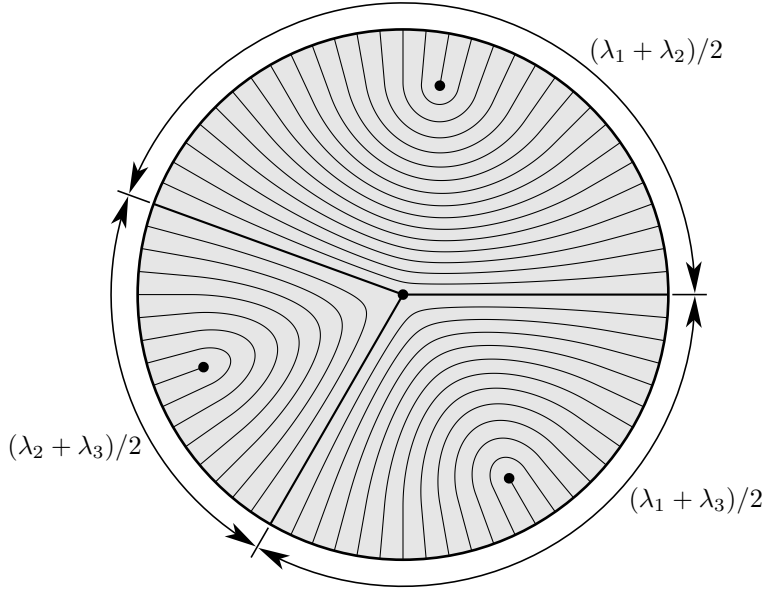}\put(-50,210){$(\lambda_1+\lambda_2)/2$}\put(-270,60){$(\lambda_2+\lambda_3)/2$}\put(-35,40){$(\lambda_1+\lambda_3)/2$}
\caption{The foliation~$\mathscr F_\lambda^0$}\label{foli-fig}
\end{figure}

\begin{remark}
The foliation~$\mathscr F_\lambda^0$ is designed to have a single $3$-prong singularity and three $1$-prong singularities.
Any measured foliation on~$\mathbb RP^2$ with such number and types of
singularities is isomorphic, after a measure rescaling, to~$\mathscr F_\lambda^0$
for some~$\lambda\in\Delta$, since one can always find a closed transversal intersecting all separatrices and
representing the non-trivial element of~$\pi_1(\mathbb RP^2)$. By cutting along such a transversal
one obtains a foliation on a $2$-disc of the form shown in Figure~\ref{foli-fig}.
\end{remark}

This foliation generalizes the one constructed in~\cite{AY}
for the purpose of providing an example of a stable foliation
of a pseudo-Anosov homeomorphism of~$\mathbb RP^2$ whose parameters are not
quadratic irrationals.

\begin{figure}[ht]
\includegraphics[scale=.7]{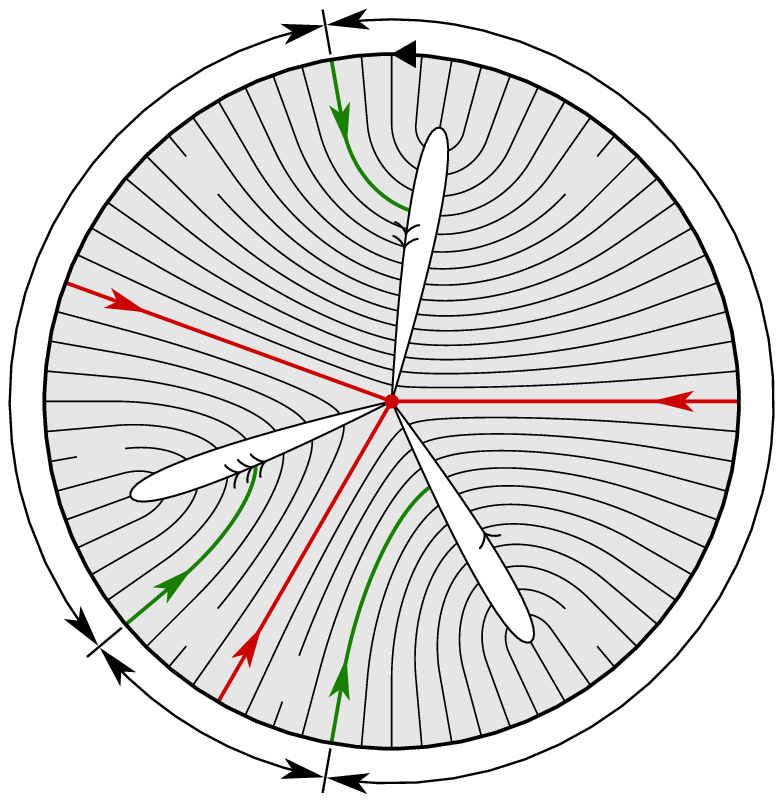}\put(-6,90){$\lambda_1$}\put(-165,120){$\lambda_2$}
\put(-130,10){$\lambda_3$}\put(-48,36){$1$}\put(-48.5,128){$2$}\put(-125,127){$3$}\put(-145,71){$4$}
\put(-125,36){$5$}\put(-107,25){$6$}\put(-161,80){$\gamma$}
\includegraphics[scale=.7]{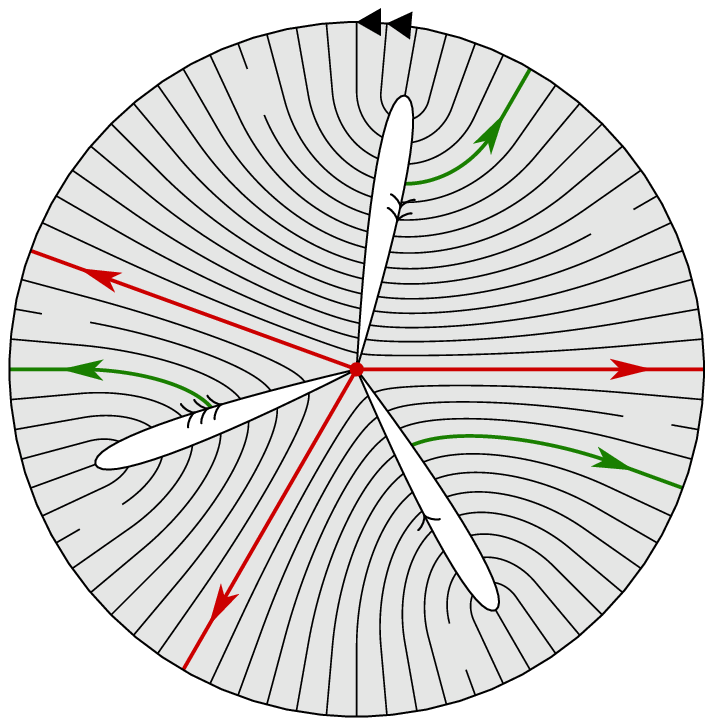}\put(-66,25){$1$}\put(-28,71){$6$}\put(-34.5,111){$3$}
\put(-107,137){$2$}\put(-145,91){$5$}\put(-138,52){$4$}

\includegraphics[scale=.7]{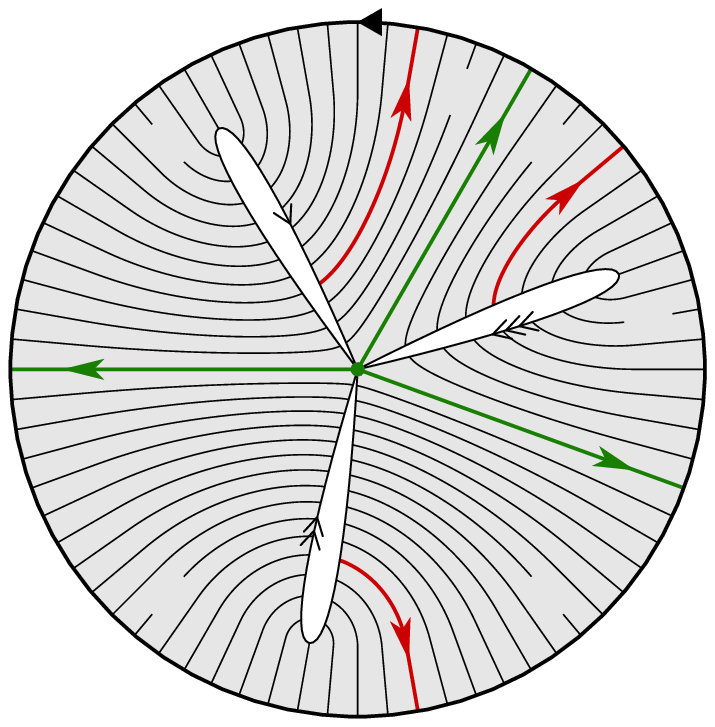}\put(-162,80){$\gamma$}\put(-125,36){$1$}\put(-49,36){$4$}
\put(-28,91.5){$3$}\put(-48.5,127){$6$}\put(-67,137.5){$5$}\put(-125,127){$2$}
\includegraphics[scale=.7]{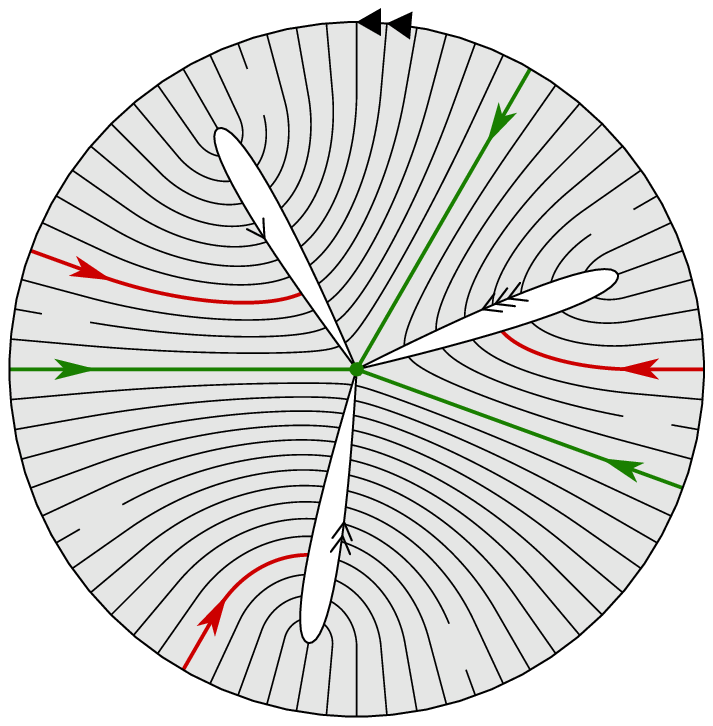}\put(-66,25){$1$}\put(-28,71){$6$}\put(-34.5,111){$3$}
\put(-107,137){$2$}\put(-145,91){$5$}\put(-138,52){$4$}

\caption{The surface~$\Sigma$ with the foliation~$\mathscr F_\lambda$}\label{4-fold-cover-fig}
\end{figure}

\begin{figure}[ht]
\includegraphics[scale=.5]{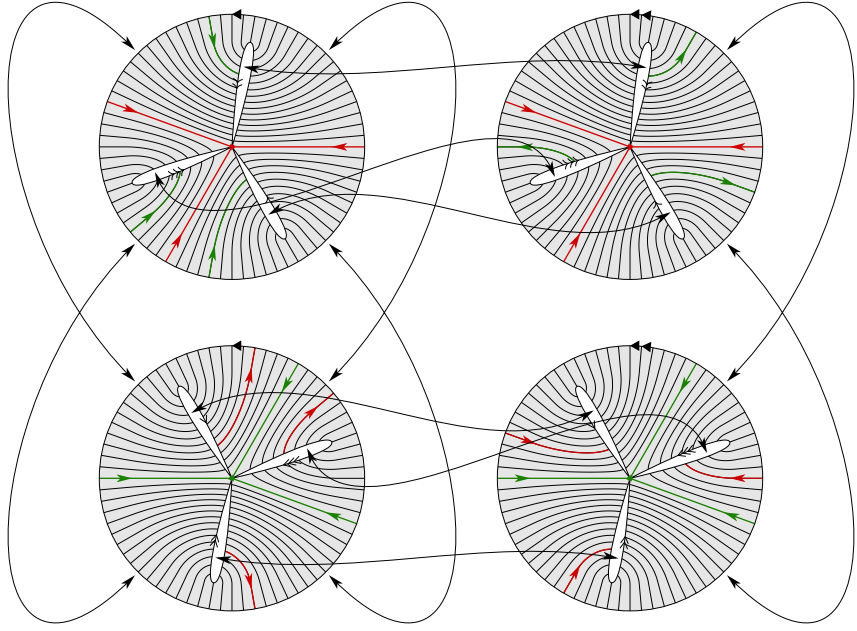}
\caption{Identifications of the boundaries in the construction of~$\Sigma_\lambda$}\label{ident-fig}
\end{figure}

Let~$\pi:\Sigma_\lambda\rightarrow\mathbb RP^2$ be a branched covering map with the smallest possible
genus of~$\Sigma_\lambda$, such that~$\Sigma_\lambda$ is an orientable surface, and the pullback~$\mathscr F_\lambda$ of~$\mathscr F_\lambda^0$
on~$\Sigma_\lambda$ is an orientable singular foliation. One can see that~$\pi$ must be four-fold (outside of the branching points
it is a regular ($\mathbb Z_2\times\mathbb Z_2$)-cover) and should have
eight simple branching points, which are preimages of the singularities of~$\mathscr F_\lambda^0$.
Six of these branching points, those that correspond to the $1$-prong singularities of~$\mathscr F_\lambda^0$,
are regular points of~$\mathscr F_\lambda$, and the other two are double (monkey) saddles of~$\mathscr F_\lambda$.
The genus of the surface~$\Sigma_\lambda$ is, therefore, equal to~$3$.

The surface~$\Sigma_\lambda$ with the foliation~$\mathscr F_\lambda$ cut into four identical pieces
is shown in Figure~\ref{4-fold-cover-fig}. The outer boundaries of the left
two pieces are identified with one another
forming a closed transversal, which is denoted by~$\gamma$.
The other boundaries are identified whenever they are indicated with the same
shape arrowheads in Figure~\ref{4-fold-cover-fig}. The identifications
are marked more explicitly in Figure~\ref{ident-fig}.

Shown in green and red colors in Figures~\ref{4-fold-cover-fig} and~\ref{ident-fig} are the segments
of the separatrices between the respective singularity and their first intersections
with~$\gamma$. These segments cut the surface into six bands
marked with numbers~$1,2,\ldots,6$ in Figure~\ref{4-fold-cover-fig}.

Recall that in the proof of Proposition~\ref{AR-minimal-prop}, for any~$\lambda\in\Delta$, we introduced a surface with
a measures singular foliation~$(\Sigma_\lambda^{\mathrm{AR}},\mathscr F_\lambda^{\mathrm{AR}})$
such that the Poincar\'e map induced by the foliation on a certain closed transversal is equivalent to~$T_\lambda^{\mathrm{AR}}$.

\begin{proposition}\label{ar=rp2-prop}
For any~$\lambda\in\Delta\setminus\partial\Delta$,
there is an isomorphism of surfaces endowed with a singular foliation~$(\Sigma^{\mathrm{AR}}_\lambda,\mathscr F^{\mathrm{AR}}_\lambda)\rightarrow
(\Sigma_\lambda,\mathscr F_\lambda)$.
\end{proposition}

\begin{proof}
The claim follows from the fact that the first return map induced on~$\gamma$ by~$\mathscr F_\lambda$
coincides with~$T_\lambda^{\mathrm{AR}}$ (after a proper identification
between~$\gamma$ and~$\mathbb S^1$), which
can be learned directly from Figure~\ref{4-fold-cover-fig}.
\end{proof}

\section{Connection to Novikov's problem}\label{Novikovproblem}
The foliations~$\mathscr F_\lambda\cong\mathscr F^{\mathrm{AR}}_\lambda$, $\lambda\in\Delta$, are also topologically
equivalent to those that arose in~\cite{DD} in the course of study
Novikov's problem on electron's quasimomentum trajectories.
We now describe the construction from~\cite{DD}.

Denote by~$\widehat\Sigma^{\mathrm{PL}}$ the union of the following squares in~$\mathbb R^3$:
$$\begin{aligned}\{i\}\times[j,j+1]\times[k,k+1],\\
[j,j+1]\times\{i\}\times[k,k+1],\\
 [j,j+1]\times[k,k+1]\times\{i\},
 \end{aligned}$$
where $i,\,j,\,k\in\mathbb Z,\ j+k\equiv1(\mathrm{mod}\,2)$.
\begin{figure}[ht]
\centerline{\includegraphics[scale=0.6]{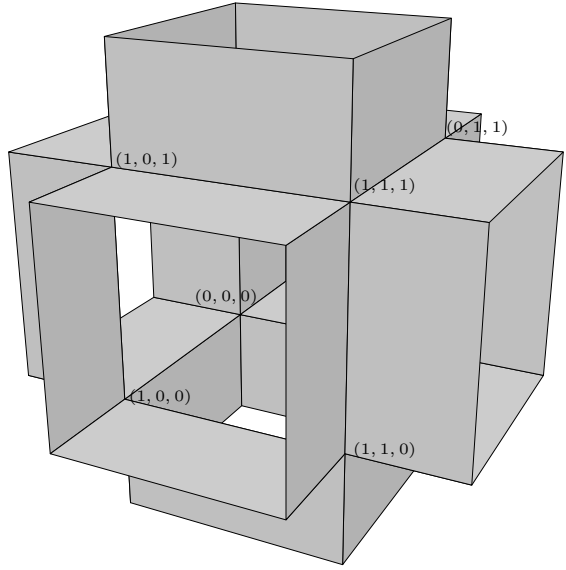}\put(-155,108){\tiny $(0,0,0)$}\put(-180,70){\tiny $(1,0,0)$}%
\put(-95,50){\tiny $(1,1,0)$}\put(-95,150){\tiny $(1,1,1)$}\put(-185,160){\tiny $(1,0,1)$}\put(-60,172){\tiny $(0,1,1)$}}
\caption{A fundamental domain of $\widehat\Sigma^{\mathrm{PL}}$}\label{surface-fundam-fig}
\end{figure}

One can see that~$\widehat\Sigma^{\mathrm{PL}}$ is a PL-surface invariant
under translations by vectors from~$2\mathbb Z^3\subset\mathbb R^3$.
Figure~\ref{surface-fundam-fig} shows a fundamental domain of~$\widehat\Sigma^{\mathrm{PL}}$
for this action. We define~$\Sigma^{\mathrm{PL}}$ to be the quotient surface~$\widehat\Sigma^{\mathrm{PL}}/(2\mathbb Z^3)$
embedded in~$\mathbb T^3=\mathbb R^3/(2\mathbb Z^3)$.

For~$\lambda\in\Delta$, we define~$\widehat{\mathscr F}^{\mathrm{PL}}_\lambda$ to be the foliation on~$\widehat\Sigma^{\mathrm{PL}}$
induced by the restriction~$\omega_\lambda=\Omega_\lambda|_{\widehat\Sigma^{\mathrm{PL}}}$ of the one-form
$$\Omega_\lambda=(\lambda_2+\lambda_3)\,dx_1+(\lambda_3+\lambda_1)\,dx_2+(\lambda_1+\lambda_2)\,dx_3,$$
where~$x_1,x_2,x_3$ are the standard coordinates on~$\mathbb R^3$. This foliation
is clearly invariant under the action of~$2\mathbb Z^3$, so we define~$\mathscr F^{\mathrm{PL}}_\lambda$ to
be the projection of~$\widehat{\mathscr F}^{\mathrm{PL}}_\lambda$ to~$\Sigma^{\mathrm{PL}}$.

\begin{proposition}\label{pl=ar-prop}
For any~$\lambda\in\Delta\setminus\{(1,0,0),(0,1,0),(0,0,1)\}$,
the foliation~$\mathscr F^{\mathrm{PL}}_\lambda$ on~$\Sigma^{\mathrm{PL}}$
is isomorphic to the foliation~$\mathscr F_\lambda$ on~$\Sigma_\lambda$, that is,
there is a homeomorphism~$\Sigma^{\mathrm{PL}}\rightarrow \Sigma_\lambda$
taking~$\mathscr F^{\mathrm{PL}}_\lambda$ to~$\mathscr F_\lambda$.
\end{proposition}

\begin{proof}
The surface~$\widehat\Sigma^{\mathrm{PL}}$ consists of square faces of three different directions,
with the interior of each foliated by straight line intervals as shown in Figure~\ref{squares-fig},
where the transverse measure of the sides of each face is also indicated.
\begin{figure}[ht]
\centerline{\includegraphics[scale=.8]{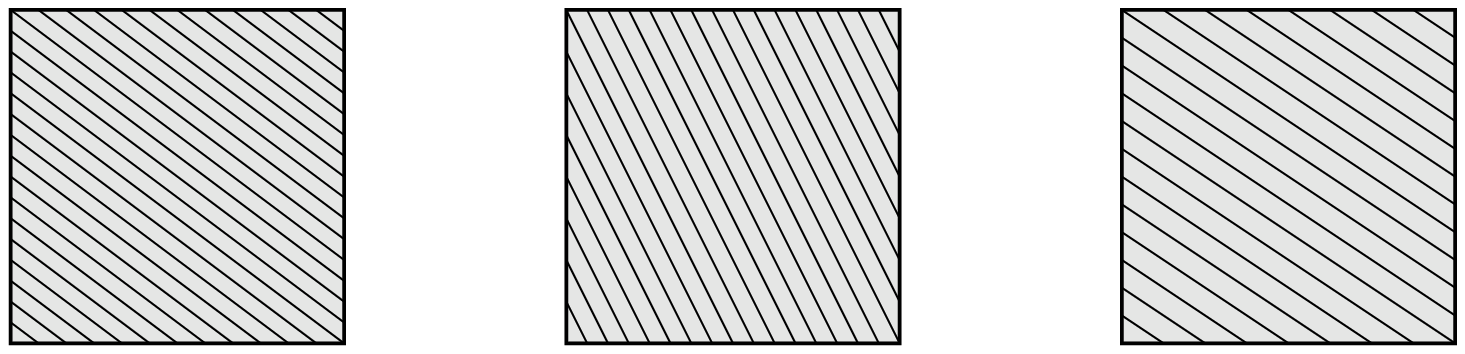}
\put(-100,0){$\scriptstyle(j,k,i)$}\put(-20,0){$\scriptstyle(j+1,k,i)$}
\put(-100,88){$\scriptstyle(j,k+1,i)$}\put(-30,88){$\scriptstyle(j+1,k+1,i)$}
\put(-55,0){$\scriptstyle\lambda_2+\lambda_3$}
\put(-110,44){$\scriptstyle\lambda_1+\lambda_3$}
\put(-228,0){$\scriptstyle(j,i,k)$}\put(-148,0){$\scriptstyle(j,i,k+1)$}
\put(-228,88){$\scriptstyle(j+1,i,k)$}\put(-158,88){$\scriptstyle(j+1,i,k+1)$}
\put(-183,0){$\scriptstyle\lambda_1+\lambda_2$}
\put(-238,44){$\scriptstyle\lambda_2+\lambda_3$}
\put(-356,0){$\scriptstyle(i,j,k)$}\put(-276,0){$\scriptstyle(i,j+1,k)$}
\put(-356,88){$\scriptstyle(i,j,k+1)$}\put(-286,88){$\scriptstyle(i,j+1,k+1)$}
\put(-311,0){$\scriptstyle\lambda_1+\lambda_3$}
\put(-366,44){$\scriptstyle\lambda_1+\lambda_2$}}
\caption{The foliation~$\widehat{\mathscr F}^{\mathrm{PL}}$ on the square faces
of~$\widehat\Sigma^{\mathrm {PL}}$; $i,\,j,\,k\in\mathbb Z,\ j+k\equiv1(\mathrm{mod}\,2)$}\label{squares-fig}
\end{figure}

After taking the quotient of
these  by~$2\mathbb Z^3$ we get 12 distinct square faces that constitute the surface~$\Sigma^{\mathrm{PL}}$.
It is a direct check that~$\Sigma^{\mathrm{PL}}$ is orientable, and so is the foliation~$\mathscr F_\lambda^{\mathrm{PL}}$.

Now there is an action of~$\mathbb Z_2\times\mathbb Z_2$ on the torus~$\mathbb T^3=\mathbb R^2/(2\mathbb Z^3)$
by the following two transformations and their composition:
\begin{equation}\label{action-eq}
x+2\mathbb Z^3\mapsto x+(1,1,1)+2\mathbb Z^3,\quad
x+2\mathbb Z^3\mapsto -x+2\mathbb Z^3.
\end{equation}
One can see that the surface~$\Sigma^{\mathrm{PL}}$ and the foliation~$\mathscr F^{\mathrm{PL}}_\lambda$
are invariant under this action.
The first transformation in~\eqref{action-eq} inverses the orientation of~$\Sigma$
whereas the second one preserves the orientation of~$\Sigma$ but flips the orientation of~$\mathscr F^{\mathrm{PL}}_\lambda$.
A~fundamental domain of~$\Sigma^{\mathrm{PL}}$ for this action can be chosen to be the union of three
square faces shown in Figure~\ref{3squares-fig}, where it is also indicated
which sides of the squares have to be identified to obtain~$\Sigma^{\mathrm{PL}}/(\mathbb Z_2\times\mathbb Z_2)$.
Figure~\ref{smoothing-fig} demonstrates that we get a surface homeomorphic to~$\mathbb RP^2$ with a foliation
topologically equivalent to~$\mathscr F_\lambda^0$.

Now the claim of the proposition follows from the fact that~$\Sigma^{\mathrm{PL}}$ is a branched orientable $4$-sheeted cover of this surface
on which the foliation also becomes orientable.
\end{proof}

\begin{figure}[ht]
\centerline{\includegraphics[scale=.7]{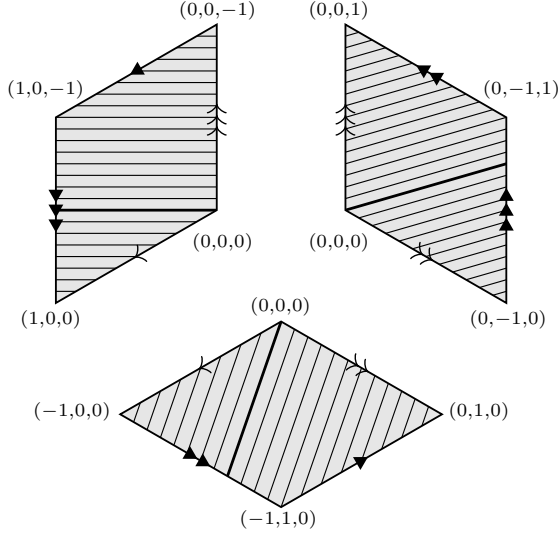}\put(-125,102){$\scriptstyle(0,0,0)$}
\put(-190,73){$\scriptstyle(1,0,0)$}\put(-130,190){$\scriptstyle(0,0,-1)$}\put(-195,160){$\scriptstyle(1,0,-1)$}
\put(-103,78){$\scriptstyle(0,0,0)$}\put(-185,38){$\scriptstyle(-1,0,0)$}\put(-107,-3){$\scriptstyle(-1,1,0)$}
\put(-28,38){$\scriptstyle(0,1,0)$}\put(-81,102){$\scriptstyle(0,0,0)$}\put(-20,73){$\scriptstyle(0,-1,0)$}
\put(-81,190){$\scriptstyle(0,0,1)$}\put(-15,160){$\scriptstyle(0,-1,1)$}}
\caption{A fundamental domain for~$\Sigma^{\mathrm PL}/(\mathbb Z_2\times\mathbb Z_2)$}\label{3squares-fig}
\end{figure}

\begin{figure}[ht]
\includegraphics[scale=.7]{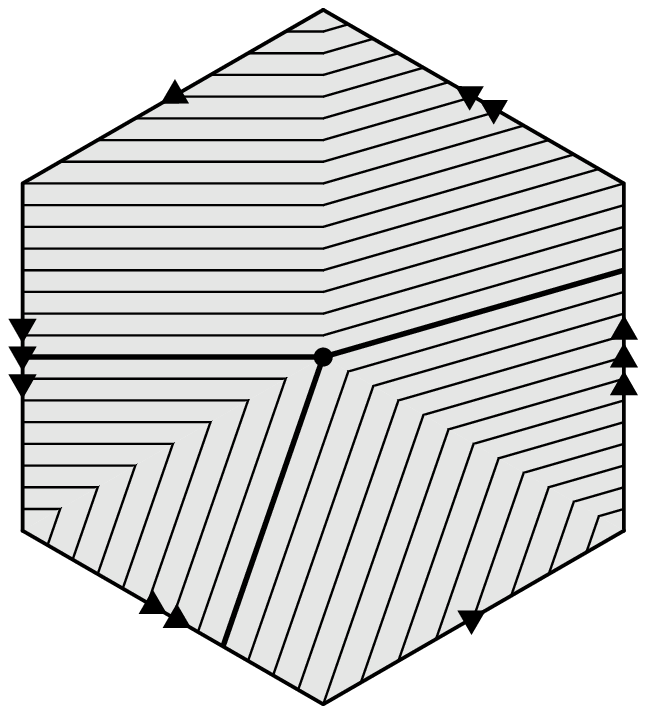}
\quad\raisebox{75pt}{$\sim$}\quad
\includegraphics[scale=.7]{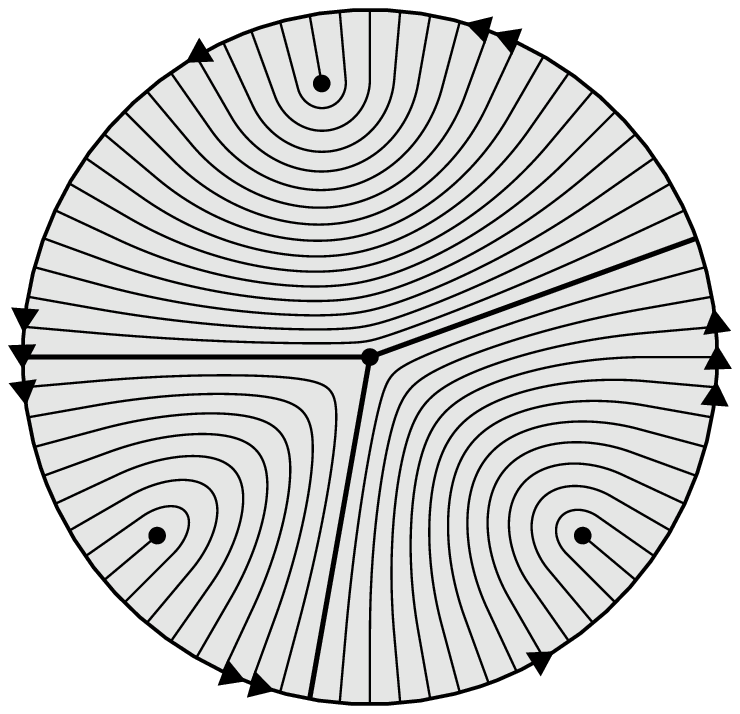}
\caption{The foliation on~$\Sigma^{\mathrm{PL}}/(\mathbb Z^2\times\mathbb Z^2)$ before (left) and after (right) smoothing}\label{smoothing-fig}
\end{figure}

\section{Systems of isometries}\label{systemsofisometries} 
The notion of a system of isometries was introduced by G.\,Levitt, D.\,Gaboriau, and F.\,Paulin in~\cite{GLP}.
A particular case appeared earlier in Levitt's paper~\cite{L}.
The concept of a system of isometries is essentially equivalent to that of a band complex (without orientation flips).
For the definition of a band complex the reader is referred to~\cite{BF}. 

For any two intervals~$I=[a,b]$ and~$J=[c,d]$ of equal length, $d-c=b-a$, we denote by~$t_{I,J}$ the translation
map from~$I$ to~$J$: $$t_{I,J}(x)=x-a+c,\quad x\in I.$$

\begin{definition}
\emph{A system of isometries} is a pair~$S=(D,X)$ in which~$D\subset\mathbb R$ is a finite union of closed intervals,
and~$X=\{\{A_1,B_1\},\ldots,\{A_n,B_n\}\}$ is a collection of unordered pairs of closed subintervals of~$D$ such that
in each pair~$\{A,B\}\in X$ the intervals~$A$ and~$B$ are of equal length. The set~$D$ is called
\emph{the support} of~$S$.

We think of~$S$ as a collection $\{t_{A,B}\}_{\{A,B\}\in X}$ of partial isometries of~$D$ and
write~$\phi\in S$ whenever~$\phi$ has the form~$t_{A,B}$ with~$\{A,B\}\in X$.
\end{definition}

\begin{definition}\label{orbit-str-graph-def}
With every system of isometries~$S$ we associate an infinite graph denoted~$\Gamma(S)$,
which we call \emph{the orbit structure graph of~$S$},
whose set of vertices is the support of~$S$ (endowed with a discrete topology),
and the edges are pairs of vertices~$\{x,y\}$ such that~$\phi(x)=y$
for some~$\phi\in S$. For any point~$x$ in the support of~$S$,
by~$\Gamma_x(S)$ (or simply~$\Gamma_x$ if~$S$ is fixed) we denote the connected
component of~$\Gamma(S)$ containing~$x$.

The graph~$\Gamma(S)$ is endowed with the path metric.
\end{definition}

\begin{definition}
\emph{The orbit~$O_x(S)$} of a system of isometries~$S=(D,X)$, \emph{passing through}~$x\in D$
is the set of vertices of~$\Gamma_x(S)$.

A system of isometries~$S=(D,X)$ is called \emph{minimal} if every orbit~$O_x(S)$ is everywhere dense in~$D$.
\end{definition}

\begin{definition}
A minimal system of isometries~$S$ is said to be of \emph{toral type} if all but finitely many
connected components of the graph~$\Gamma(S)$ are quasi-isometric to~$\mathbb R^k$
for some fixed~$k\geqslant 2$. If all but finitely many connected components of~$\Gamma(S)$
are quasi-isometric to~$\mathbb R$, and~$S$ is minimal, then~$S$ is said to be of \emph{surface type}.

A minimal system of isometries~$S$ is said to be of \emph{thin type} if every connected component
of~$\Gamma(S)$ is quasi-isometric to an infinite tree, but not quasi-isometric to a straight line.
\end{definition}

It is shown by E.Rips (see~\cite{BF}) that any minimal system of isometries is either of toral type, or surface type, or thin type.

\begin{definition}
\emph{An invariant measure} of a system of isometries~$S=(D,X)$ is a measure on~$D$ which is invariant
under all partial isometries~$\phi\in S$.
A system of isometries~$S=(D,X)$ is said to be \emph{uniquely ergodic} if the Lebesgue measure on~$D$
is the only invariant measure of~$S$, up to scale.
\end{definition}

\begin{theorem}\label{s-lambda-th}
For~$\lambda\in\Delta$, let~$S_\lambda$ be the following system of isometries:
\begin{equation}\label{S-lambda-eq}
S_\lambda=\bigl([0,1],\bigl\{\{[0,\lambda_1],[1-\lambda_1,1]\},\{[0,\lambda_2],[1-\lambda_2,1]\},\{[0,\lambda_3],[1-\lambda_3,1]\}\bigr\}\bigr).
\end{equation}
Then we have the following trichotomy.
\begin{itemize}
\item[{\rm(i)}]
If~$\lambda\notin\rg$ or~$\lambda\in\mathbb Q^3\cap\Delta$, then~$S_\lambda$ has a finite orbit and thus is not minimal.
\item[{\rm(ii)}]
If~$\lambda\in\rg_{\mathrm{irr}}$, then~$S_\lambda$ is of thin type and is uniquely ergodic.
\item[{\rm(iii)}]
If~$\lambda\in\rg\setminus\bigl(\rg_{\mathrm{irr}}\cup\mathbb Q^3)\bigr)$, then~$S_\lambda$
is of surface type and is uniquely ergodic.
\end{itemize}
\end{theorem}

\begin{proof}
(i) It is easy to see that if~$\lambda\in\mathbb Q^3$, then all orbits of~$S_\lambda$
are finite. From now on we assume that~$\lambda\not\in\mathbb Q^3$.

If~$\lambda\in\Delta_0\setminus\partial\Delta_0$, the orbit~$O_{1/2}(S_\lambda)$ consists
of a single point, which is~$1/2$.

One can show that if~$\lambda\in f_i(\Delta)$, then the transition from~$S_\lambda$ to~$S_{f_i^{-1}(\lambda)}$
is essentially equivalent to running one step of a \emph{Rips machine}. 
The reader not familiar with the concept of a Rips machine is referred to~\cite{BF,GLP}
for a detailed account.

Informally, a Rips machine can be interpreted in two different ways. Originally, it appeared as an analogue of Makanin--Razborov procedure for solving equations in groups. It can also be seen as a renormalization procedure similar in its spirit to the Rauzy induction. It consists of six kinds of moves (of which we need only two: slide and collapse from a free arc)
and results in a classification of types of isometric actions of finitely generated group on $\mathbb R$-trees or, equivalently, of the resolving band complexes and systems of isometries corresponding to them.

Now we proceed with the definition of a modification of a Rips machine. 

Let~$S=(D,X)$ be a system of isometries such that~$D=D'\cup I$, where~$I$ is an interval disjoint from~$D'$,
and there is an interval~$J\subset D'$ such that~$\{I,J\}\in X$. Denote by~$\psi$ the following map~$D\rightarrow D'$:
$$\psi(x)=\left\{\begin{aligned}&x,&&\text{if }x\in D',\\
&t_{I,J}(x),&&\text{otherwise}.
\end{aligned}\right.$$
We will say that the system of isometries~$S'=(S',X')$ is obtained from~$S$ by \emph{a simplification} if
$$X'=\{\psi(I'),\psi(J')\}_{\{I',J'\}\in X,\,\{I',J'\}\ne\{I,J\}}.$$
Clearly, the map~$\psi$ extends to a quasi-isometry~$\Gamma(S)\rightarrow\Gamma(S')$,
so, a simplification does not change the dynamical properties of the system. Any simplification can
be decomposed into a sequence of several slides and a collapse from a free arc.

One can see that the addition of simplifications to the Rips machinery does not affect the output (toral type, surface type,
or thin type). This is due to the fact that
any transformation of a system of isometries obtained by running finitely many steps of a Rips machine with simplifications
can also be obtained by running an ordinary Rips machine \emph{followed} by a few simplifications.

In the case of~$S_\lambda$ with~$\lambda_1>\lambda_2>\lambda_3$ (other cases are
similar), we have a free arc~$(\lambda_1,1-\lambda_2)$.
A collapse from this free arc replaces~the pair~$\{[0,\lambda_1],[1-\lambda_1,1]\}$
with the following two pairs:
$$\{[0,\lambda_1-\lambda_2-\lambda_3],[\lambda_2+\lambda_3,\lambda_1]\},\
\{[\lambda_1-\lambda_2,\lambda_1],[1-\lambda_2,1]\},$$
and the new system is supported on~$[0,\lambda_1]\cup[1-\lambda_2,1]$.
By applying a simplification, we get the following system:
$$\bigl([0,\lambda_1],\bigl\{\{[0,\lambda_1-\lambda_2-\lambda_3],[1-\lambda_1,\lambda_1]\},
\{[0,\lambda_2],[\lambda_1-\lambda_2,\lambda_1]\},
\{[0,\lambda_3],[\lambda_1-\lambda_3,\lambda_1]\}
\bigr\}\bigr),$$
which is~$S_{f_1^{-1}(\lambda)}$ scaled~$\lambda_1$ times.

Thus, if~$\lambda\in f_i(\Delta)$, then the contraction map~$x\mapsto\lambda_ix$, $x\in[0,1]$,
extends to a quasi-isometry~$\Gamma(S_\lambda)\rightarrow\Gamma\bigl(S_{f_i^{-1}(\lambda)}\bigr)$.
Therefore, $S_\lambda$ has a finite orbit if and only if so has~$S_{f_i^{-1}(\lambda)}$.

If~$\lambda\notin\rg$, then there is a finite word~$\slbf i=i_1i_2\ldots i_k$
such that~$f_{\slbf i}^{-1}(\lambda)\in\Delta_0\setminus\partial\Delta_0$. Hence, $S_\lambda$ has a finite orbit.

(ii) Let~$\lambda\in\rg_{\mathrm{irr}}$, and let~$\slbf i=i_1i_2\ldots$ be the directing word for~$\lambda$.
As was just said, the transition from~$S_\lambda$ to~$S_{f_{i_1}^{-1}(\lambda)}$
is equivalent to running one step of a (inessentially modified) Rips machine, and this step includes
a collapse from a free arc. Thus, a Rips machine started from~$S_\lambda$
runs infinitely long, using collapses from a free arc and producing,
at every step, a system of isometries of the form~$S_{\lambda'}$,
$\lambda'\in\Delta\setminus\partial\Delta$. This implies that~$S_\lambda$ is minimal and of thin type.

Now let~$\mu$ be an invariant probability measure of~$S_\lambda$. Since
the system~$S_\lambda$ is minimal, the measure~$\mu$ is determined from
the following three parameters:
$$\lambda_1'=\mu([0,\lambda_1]),\
\lambda_2'=\mu([0,\lambda_2]),\
\lambda_3'=\mu([0,\lambda_3]),$$
and there is a homeomorphism~$[0,1]\rightarrow[0,1]$ that takes~$S_\lambda$
to the system
$$S'=\bigl([0,1],\bigl\{\{[0,\lambda_1'],[1-\lambda_1',1]\},\{[0,\lambda_2'],[1-\lambda_2',1]\},\{[0,\lambda_3'],[1-\lambda_3',1]\}\bigr\}\bigr),$$
with~$\Gamma(S')$ isometric to~$\Gamma(S_\lambda)$. This implies two things. First, $S$ is balanced
in the sense that~$\lambda_1'+\lambda_2'+\lambda_3=\mu([0,1])=1$,
or, equivalently, $\lambda'=(\lambda_1',\lambda_2',\lambda_3')\in\Delta$. Indeed, if~$\lambda_1'+\lambda_2'+\lambda_3<1$,
then~$\Gamma(S')$ would have compact connected components, and if~$\lambda_1'+\lambda_2'+\lambda_3>1$,
then~$\Gamma(S')$ would have non-simply connected components, neither of which holds for~$\Gamma(S_\lambda)$.
Thus,~$S'=S_{\lambda'}$.

Second, there is a Rips machine
running on~$S'=S_{\lambda'}$ exactly the same way (from the topological point of view)
as it does for~$S_\lambda$. This means that
the directing word for~$\lambda$ is also a directing word for~$\lambda'$, which implies~$\lambda'=\lambda$
due to Proposition~\ref{regular-rauzy-prop}. Therefore, $\mu$ coincides with the Lebesgue measure on~$[0,1]$.

(iii) Now suppose that~$\lambda\in\rg\setminus\rg_{\mathrm{irr}}$.
By Proposition~\ref{regular-rauzy-prop}, there exists a finite word~$\slbf i=i_1i_2\ldots i_k$
such that~$f_{\slbf i}^{-1}(\lambda)\in\partial\Delta$. This means that a Rips machine
started from~$S_\lambda$ produces a system~$S_{\lambda'}$ with
one of the parameters~$\lambda'_i$ equal to zero. Such a system is
equivalent to an irrational rotation, which is uniquely ergodic and of surface type.
\end{proof}

\begin{remark}
The ergodicity of $S_\lambda$ is equivalent to the ergodicity of the symbolic system defined by the corresponding Arnoux-Rauzy word. 
The ergodicity of this system is a corollary of a theorem of Boshernitzan~\cite{B}.
\end{remark}

\section{Double suspension surface}\label{double-susp-sec}

IETs can be viewed as a special case of systems of isometries.
There is, however, a relation of an opposite kind between the two subjects,
which was used implicitly in~\cite{D08,DS} to construct examples of chaotic regimes in Novikov's
problem and to study their properties. Namely, with every system of isometries
endowed with a little bit more combinatorial data, one associates naturally an orientable
surface with an orientable measured foliation, which gives rise to an IET
once a transversal intersecting all leaves has been chosen. As we will see below, in the case
of the system~$S_\lambda$, this general construction produces the foliation~$\mathscr F_\lambda\cong\mathscr F^{\mathrm{AR}}_\lambda
\cong\mathscr F^{\mathrm{PL}}_\lambda$
defined in the previous sections.

We now proceed with the description of this construction.

\begin{definition}
By \emph{an enhancement} of a system of isometries $S=(D,X)$ with
$$X=\bigl\{\{I_1,I_2\},\ldots,\{I_{2k-1},I_{2k}\}\bigr\}$$
we mean a circular ordering of the set~$\{I_1,I_2,\ldots,I_{2k}\}$.
To represent such a circular ordering we pick an injective map~$\theta:\{I_1,I_2,\ldots,I_{2k}\}\rightarrow\mathbb S^1$
such that the points~$\theta(I_i)$ follow on~$\mathbb S^1$ in the respective order.
The pair~$(S,\theta)$ will then be called \emph{an enhanced system of isometries}.
\end{definition}

With every enhanced system of isometries~$(S,\theta)$ we associate a surface~$\Sigma=\Sigma(S,\theta)$,
called \emph{the double suspension surface of~$(S,\theta)$}, endowed
with a singular measured foliation~$\mathscr F=\mathscr F(S,\theta)$ as follows.

Let~$S=\bigl(D,\bigl\{\{I_1,I_2\},\ldots,\{I_{2k-1},I_{2k}\}\bigr\}\bigr)$.
Pick an~$\varepsilon>0$ smaller than half the distance between any two points~$\theta_i$, $\theta_j$,
$i\ne j$, where we denote~$\theta(I_i)$ by~$\theta_i$ for brevity.
The surface~$\Sigma$ with the foliation~$\mathscr F$ is constructed in the following four steps.
\begin{description}
\item[\it Step 1]
Let~$(\Sigma_0,\mathscr F_0)$ be the Cartesian product~$D\times\mathbb S^1$ foliated
by the circles~$\{x\}\times\mathbb S^1$, $x\in D$. The leaves inherit the orientation from~$S^1$.
\item[\it Step 2]
Remove from~$\Sigma_0$ the interior of each rectangle~$I_i\times[\theta_i-\varepsilon,\theta_i+\varepsilon]$, $i=1,\ldots,2k$,
to obtain~$\Sigma_1$;
\item[\it Step 3]
For each~$i=1,\ldots,k$, identify the straight line segment~$I_{2i-1}\times\{\theta_{2i-1}-\varepsilon\}$ in~$\Sigma_1$
with~$I_{2i}\times\{\theta_{2i}+\varepsilon\}$, and $I_{2i-1}\times\{\theta_{2i-1}+\varepsilon\}$
with~$I_{2i}\times\{\theta_{2i}-\varepsilon\}$ using the map~$t_{I_{2i-1},I_{2i}}\times h$,
where~$h$ takes~$\theta_{2i-1}\pm\varepsilon$ to~$\theta_{2i}\mp\varepsilon$, to obtain~$\Sigma_2$.
\item[\it Step 4]
In~$\Sigma_2$, collapse to a point:
\begin{itemize}
\item
every circle of the form~$\{x\}\times\mathbb S^1$, $x\in\partial D$;
\item
every arc of the form~$\{x\}\times[\theta_i-\varepsilon,\theta_i+\varepsilon]$, $i=1,\ldots,2k$, $x\in\partial I_i$,
\end{itemize}
to obtain~$\Sigma$ (some of these points may coincide due to the identifications at the previous step).
The foliation~$\mathscr F$ is inherited from~$\mathscr F_0$.
\end{description}

We are particularly interested in the double suspension surface of the system~$S_\lambda$
defined by~\eqref{S-lambda-eq} and
endowed with the following cyclic ordering of the involved intervals:
\begin{equation}\label{interval-ordering-eq}
[0,\lambda_1],\ [0,\lambda_2],\ [0,\lambda_3],\ [1-\lambda_1,1],\ [1-\lambda_2,1],\ [1-\lambda_3,1].
\end{equation}
The respective enhancement of~$S_\lambda$ is denoted by~$\theta_\lambda$.
The double suspension surface of~$(S_\lambda,\theta_\lambda)$ is shown in Figure~\ref{double-susp-fig}, where
arcs with the same markings are identified.
\begin{figure}[ht]
\includegraphics[scale=.4]{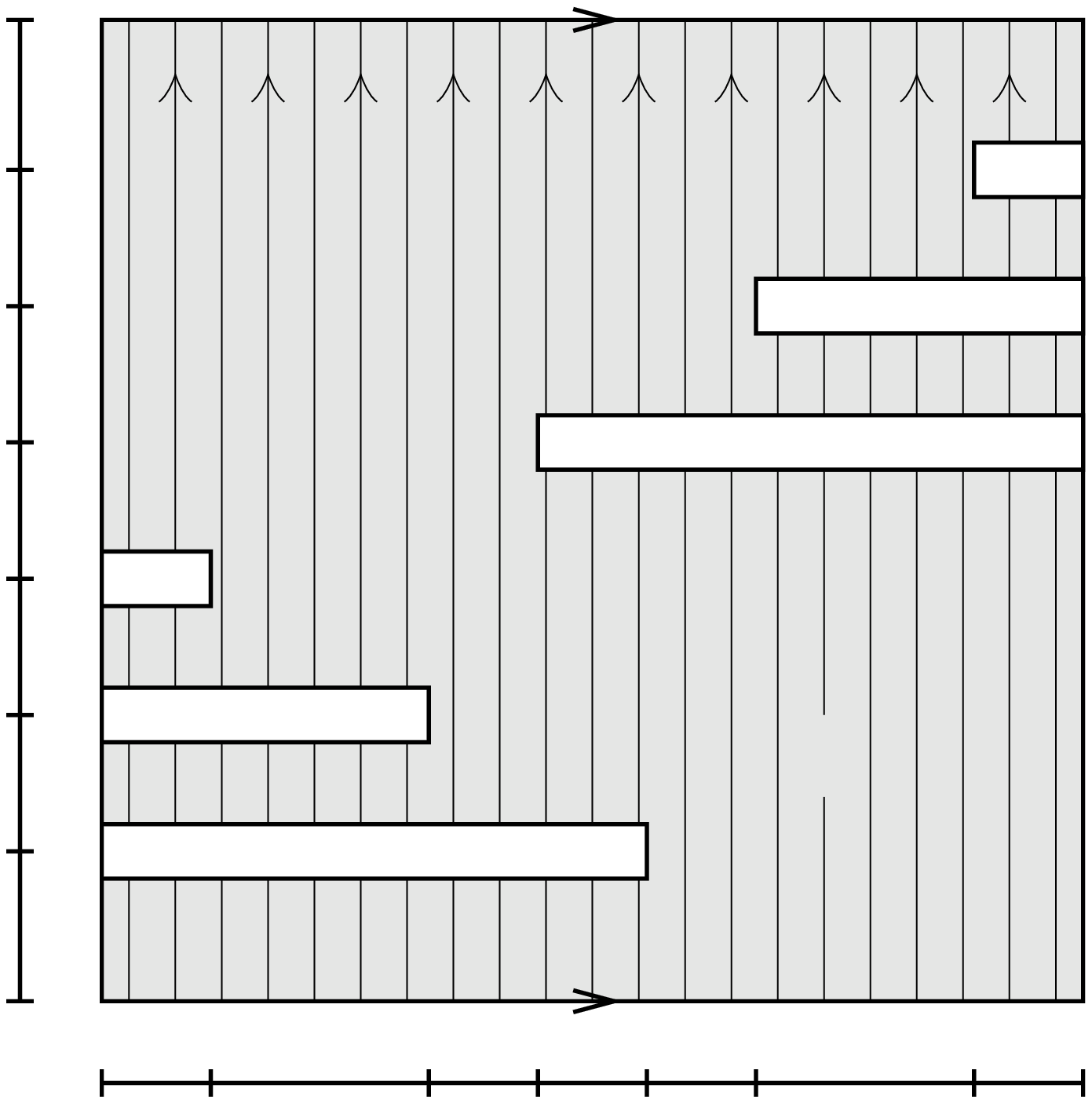}
\put(-150.5,0){$\scriptstyle0$}\put(-136,0){$\scriptstyle\lambda_3$}\put(-105,0){$\scriptstyle\lambda_2$}
\put(-92,0){$\scriptstyle1{-}\lambda_1$}\put(-72,0){$\scriptstyle\lambda_1$}
\put(-59,0){$\scriptstyle1{-}\lambda_2$}\put(-28,0){$\scriptstyle1{-}\lambda_3$}
\put(-6,0){$\scriptstyle1$}
\put(-171,41){$\scriptstyle\theta_1$}\put(-171,61){$\scriptstyle\theta_2$}
\put(-171,81){$\scriptstyle\theta_3$}\put(-171,101){$\scriptstyle\theta_4$}
\put(-171,121){$\scriptstyle\theta_5$}\put(-171,141){$\scriptstyle\theta_6$}
\put(-171,19){$\scriptstyle0$}\put(-171,163){$\scriptstyle1$}
\put(-47,53){$\Sigma_1$}
\hskip1cm
\includegraphics[scale=.4]{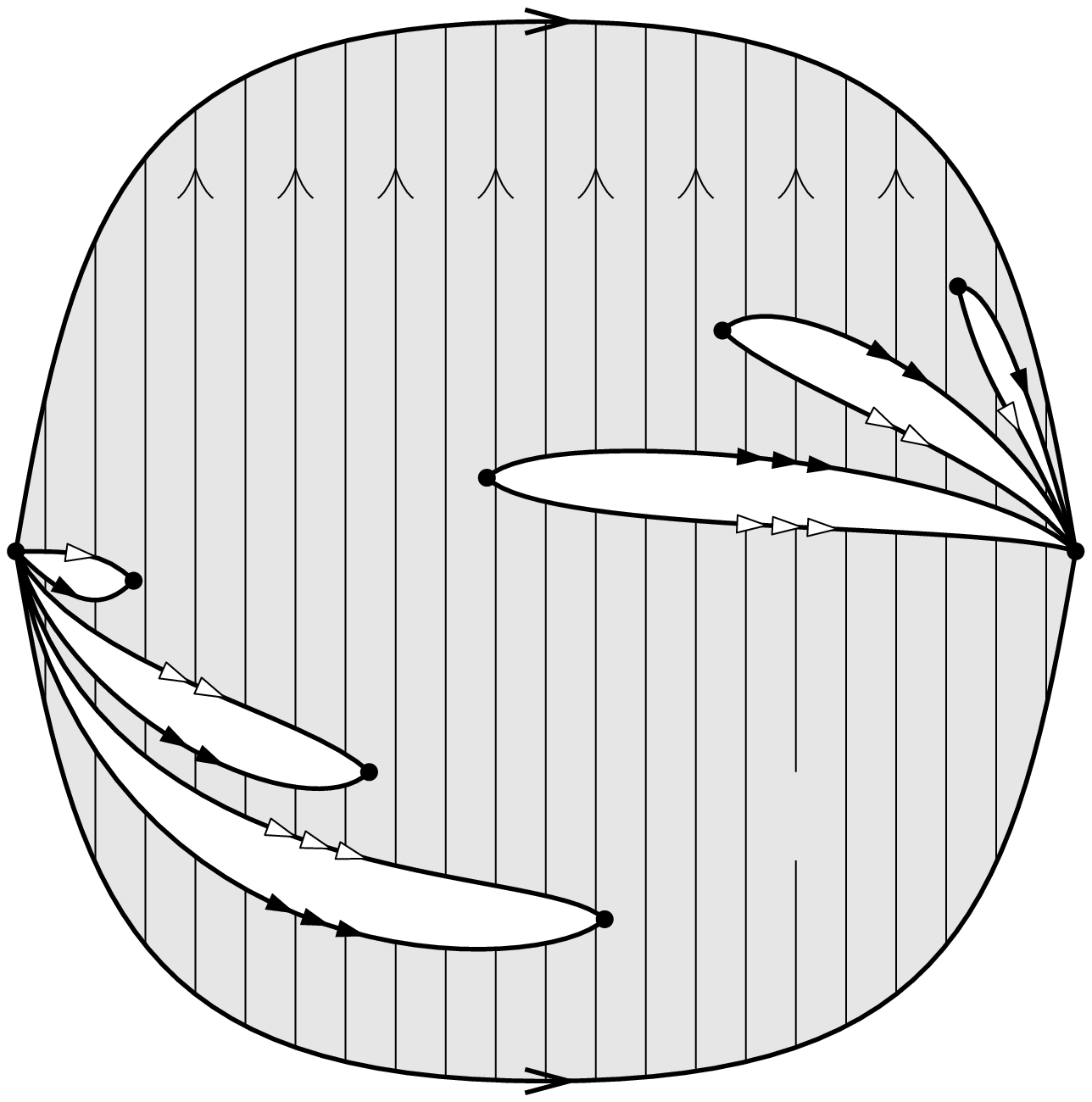}\put(-45,53){$\Sigma$}
\caption{Double suspension surface of~$(S_\lambda,\theta_\lambda)$}\label{double-susp-fig}
\end{figure}

\begin{proposition}\label{double=ar-prop}
The double suspension surface~$(\Sigma,\mathscr F)$ of~$(S_\lambda,\theta_\lambda)$ endowed
with the respective singular measured foliation is equivalent to~$(\Sigma_\lambda,\mathscr F_\lambda)\cong
(\Sigma^{\mathrm{AR}}_\lambda,\mathscr F^{\mathrm{AR}}_\lambda)$ \emph(after multiplying by two
the transverse measure of the latter\emph).
\end{proposition}

\begin{proof}
It suffices to find a transversal~$\gamma$ for~$\mathscr F$ the Poincar\'e map on which coincides
with~$T^{\mathrm{AR}}_\lambda$ (after a suitable identification of~$\gamma$ with~$\mathbb S^1$).
Such a transversal is shown in Figure~\ref{sigma-trans-fig}, where the six strips into which the separatrices---%
continued to the first intersection with~$\gamma$---cut the surface are also indicated (we use the~$\Sigma_2$-stage
of the construction of~$\Sigma$ in the picture).
\begin{figure}[ht]
\includegraphics[scale=.4]{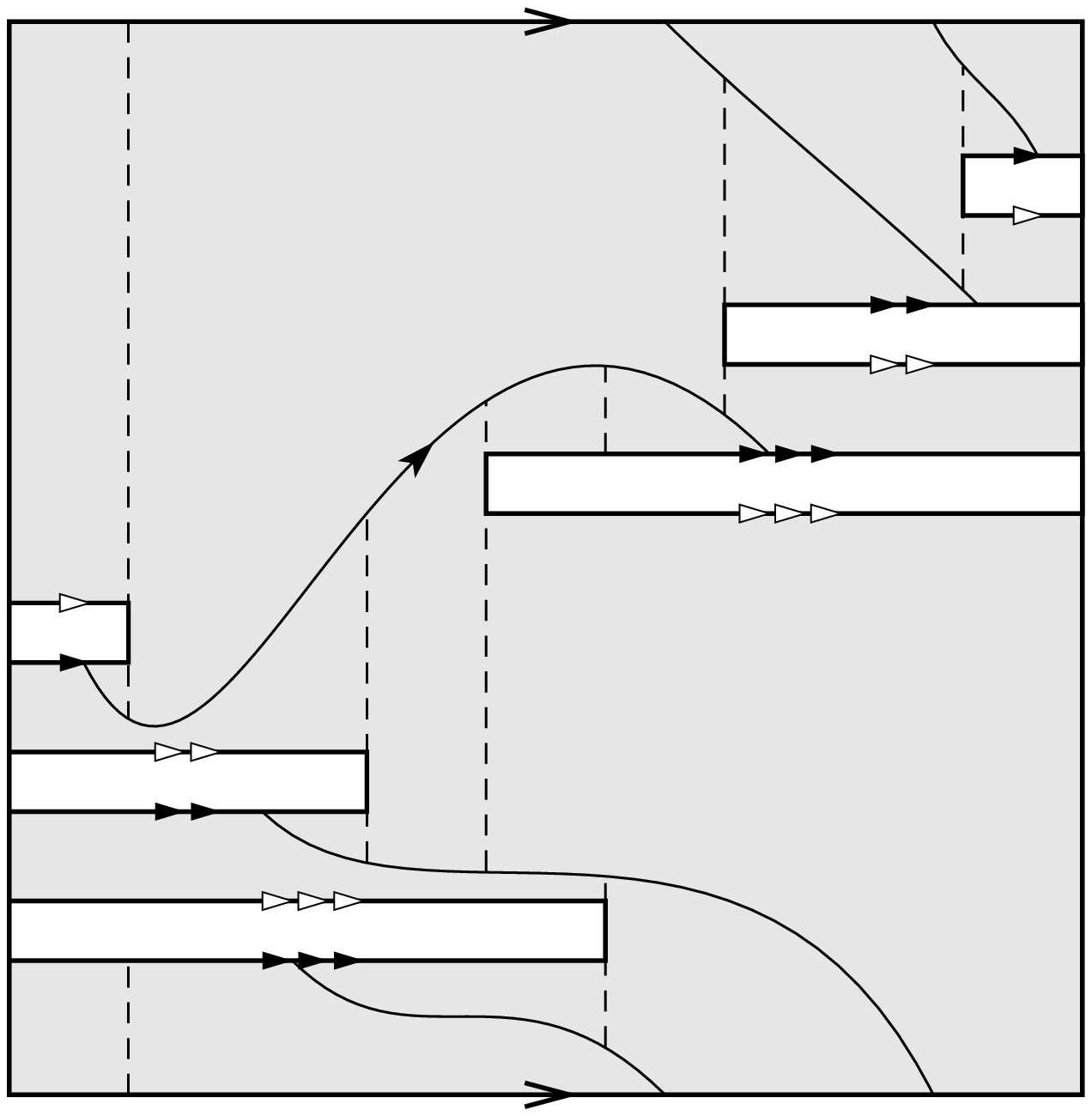}
\put(-95,120){$2$}\put(-110,8){$2$}\put(-65,92){$2$}
\put(-30,93){$3$}\put(-90,16){$\scriptstyle3$}\put(-115,55){$3$}
\put(-19,132){$\scriptstyle3$}\put(-53,12){$4$}\put(-35,135){$4$}
\put(-14,113){$5$}\put(-105,37){$\scriptstyle5$}
\put(-142,10){$5$}\put(-142,110){$5$}\put(-78,92){$5$}
\put(-95,60){$6$}\put(-45,50){$1$}\put(-120,33){$1$}
\put(-42,115){$1$}\put(-13,138){$1$}\put(-136,57){$\scriptstyle1$}
\put(-100,93){$\gamma$}
\caption{The transversal~$\gamma$ on~$\Sigma_\lambda$ on which
the first return map is~$T^{\mathrm{AR}}_\lambda$}\label{sigma-trans-fig}
\end{figure}
It is a direct check that the Poincar\'e map on~$\gamma$ is, indeed, $T^{\mathrm{AR}}_\lambda$ scaled twice.
\end{proof}

\section{Unique ergodicity of~$S_\lambda$ and~$T_\lambda^{\mathrm{AR}}$}

According to Theorem~\ref{s-lambda-th}, the system of isometries~$S_\lambda$ is uniquely ergodic whenever~$\lambda\in\rg_{\mathrm{irr}}$.
However, this does not imply the unique ergodicity of the foliation~$\mathscr F_\lambda$
(equivalently, of the IET~$T^{\mathrm{AR}}_\lambda$). Moreover, we have seen above
in Proposition~\ref{non-uniquely-ergodic-prop} that there exist uncountably many~$\lambda\in\rg_{\mathrm{irr}}$
such that~$\mathscr F_\lambda$ is \emph{not} uniquely ergodic. Here we discuss the nature of this discrepancy, which
may also occur for more general system of isometries.

Let~$(S,\theta)$ be an arbitrary enhanced system of isometries with
$$S=(D,X),\quad X=\bigl\{\{I_1,I_2\},\ldots,\{I_{2k-1},I_{2k}\}\bigr\},$$
and let~$(\Sigma,\mathscr F)$ be the double suspension surface of~$(S,\theta)$
endowed with the respective singular measured foliation.

In what follows we make the following two assumptions on~$S$ (which hold for~$S_\lambda$ when~$\lambda\in\rg_{\mathrm{irr}}$):
\begin{enumerate}
\item[(A1)]
$S$ is \emph{balanced}, which means that the total length of the intervals~$I_i$, $i=1,2,\ldots,2k$, is equal to the total
length of~$D$:
$$\sum_{i=1}^{2k}|I_i|=|D|;$$
\item[(A2)]
$S$ is of thin type.
\end{enumerate}

\begin{theorem}\label{ergodicity-dim-th}
Under the assumptions made above, if~$S$ is uniquely ergodic, then~$\mathscr F$ has at most two
ergodic components. If~$\mathscr F$ is uniquely ergodic, then so is~$S$.

More generally,
let~$\mathscr M(S)$ be the space of all invariant
Borel measures
of~$S$. Then
\begin{equation}
\dim\mathscr M(S)\leqslant\dim\mathscr M(\mathscr F)\leqslant2\dim\mathscr M(S).
\end{equation}

\end{theorem}

We need some preparations to prove this theorem. They will also allow us to formulate a sufficient
condition for the equality $\dim\mathscr M(S)=\dim\mathscr M(\mathscr F)$.

Under the assumptions above,
there is a natural inclusion~$\mathscr M(S)\hookrightarrow\mathscr M(\mathscr F)$, which
we now describe.

We use the notation from Section~\ref{double-susp-sec}. For~$i\in\{1,2,\ldots,2k\}$, let~$\alpha_i\subset\Sigma$
be the image of the straight line segment~$I_i\times\{\theta_i-\varepsilon\}$
under the identifications made on Step~3 of the construction of~$\Sigma$.
By construction, $\Sigma$ is obtained from a subset of~$D\times\mathbb S^1$ by
identifications. The projection of this subset to the first factor
gives rise to a map~$p:\Sigma\rightarrow D$ which is multivalued on~$\bigcup_{i=1}^{2k}\alpha_i$.

Clearly, the preimage~$p^{-1}(X)$ of any invariant subset~$X\subset D$ of~$S$ is
a union of leaves of~$\mathscr F$.
Therefore, the pullback of any invariant measure~$\mu\in\mathscr M(S)$ under~$p$
is a transverse invariant measure on~$(\Sigma,\mathscr F)$, which we denote by~$\mu^\dag$.
The following statement is also obvious.

\begin{lemma}
The map~$\mathscr M(S)\rightarrow\mathscr M(\mathscr F)$ that takes any~$\mu\in\mathscr M(S)$
to~$\mu^\dag$ is injective.
\end{lemma}

Now we recall some general facts about the orbit structure graph~$\Gamma(S)$.

\begin{proposition}\label{almost-all-trees-prop}
If~$S$ is balanced and of thin type, then all but finitely many components of~$\Gamma(S)$
are infinite trees with at most two topological ends.
\end{proposition}

The proof can be found in~\cite{G} or \cite{BF}.

\begin{definition}
A subset~$X\subset D$ is called \emph{negligible} if~$\mu(X)=0$
for any~$\mu\in\mathscr M(S)$.

A union of leaves of~$\mathscr F$ is called \emph{negligible} if
its intersection with any transversal has zero measure with respect
to any measure from~$\mathscr M(\mathscr F)$.
\end{definition}

Let~$D_1$ (respectively, $D_2$) be the set of~$x\in D$ such the graph~$\Gamma_x(S)$
is a one-ended (respectively, two-ended) tree and~$O_x(S)$ is disjoint from the boundaries of~$I_i$, $i=1,2,\ldots,2k$.
It follows form Proposition~\ref{almost-all-trees-prop}
that the complement~$D\setminus(D_1\cup D_2)$ is countable, and hence, negligible.

\begin{lemma}\label{number-of-leaves-lem}
If~$x\in D_m$, $m=1,2$, then~$p^{-1}(O_x(S))$ consists of exactly~$m$ regular leaves of~$\mathscr F$.
\end{lemma}

\begin{proof}
Specifying the enhancement~$\theta$ of the system of isometries~$S$ turns the orbit structure graph~$\Gamma(S)$
into a \emph{ribbon graph}, that is, a graph with a cyclic ordering of half-edges at every vertex.
This is done as follows.

The half-edges of~$\Gamma(S)$ are in an obvious one-to-one correspondence with pairs~$(x,i)$,
where~$x\in D$, and~$i\in\{1,2,\ldots,2k\}$ are such that~$x\in I_i$. Two half-edges~$(x,i)$ and~$(y,j)$
form an edge if and only if~$t_{I_i,I_j}\in S$ and~$t_{I_i,I_j}(x)=y$, and this edge joins the vertices~$x$ and~$y$.
Thus, the half-edges at a vertex~$x$ are all pairs~$(x,i)$ with~$I_i\ni x$. The cyclic ordering of these
half-edges is obtained by restricting~$\theta$ to the subset~$\{i:x\in I_i\}\subset\{1,2,\ldots,2k\}$.
This restriction will be denoted by~$\theta|_x$.

Now fix~$x_*\in D_1\cup D_2$ and pick a leaf~$\gamma$ of~$\mathscr F$ such that~$\gamma\subset p^{-1}(x_*)$.
Let
$$\ldots,A_{-1},A_0,A_1,A_2,\ldots$$
be the intersection points of~$\gamma$ with~$\cup_{i=1}^{2k}\alpha_i$
numbered in the order they follow on~$\gamma$. 
For every~$i\in\mathbb Z$, the point~$A_i$
has two preimages in~$\Sigma_1$, which have the form
$$(x_i,\theta_{l_i}-\varepsilon)\quad\text{and}\quad(y_i,\theta_{m_i}+\varepsilon),$$
such that the following conditions hold:
\begin{enumerate}
\item[(C1)] $x_i\in I_{l_i}\cap O_{x_*}(S)$;
\item[(C2)] $y_i=x_{i+1}=t_{I_{l_i},I_{m_i}}(x_i)$;
\item[(C3)] the interval~$I_{l_i}$ is next to~$I_{m_{i-1}}$ with respect to the cyclic ordering~$\theta|_{x_i}$.
\end{enumerate}

Let
$$\ldots,h_{-1},h_0,h_1,h_2,\ldots$$
be the biinfinite sequence of half-edges of~$\Gamma(S)$ defined by~$h_{2i}=(x_i,l_i)$, $h_{2i+1}=(y_i,m_i)$.
The the conditions~(C1--C3) above can be reformulated as follows:
\begin{equation}\label{half-edge-seq-eq}
h_{2i}\text{ and }h_{2i+1}\text{ form an edge of }\Gamma_{x_*}(S),\text{ and }h_{2i}\text{ is next to }h_{2i-1}
\text{ w.r.t. }\theta|_{x_i}.
\end{equation}

One can see that Condition~\eqref{half-edge-seq-eq} is also sufficient for a biinfinite sequence~$(h_i)_{i\in\mathbb Z}$
to define a leaf of~$\mathscr F$ contained in~$p^{-1}(x_*)$, and another such sequence
defines the same leaf if and only if it is obtained from~$(h_i)_{i\in\mathbb Z}$ by a shift.

We now claim that the number of biinfinite sequences satisfying~\eqref{half-edge-seq-eq}, viewed
up to a shift,
is the same as the number of topological ends of~$\Gamma_{x_*}(S)$ (this holds true
for any infinite tree in place of~$\Gamma_{x_*}(S)$). Indeed, the graph~$\Gamma_{x_*}(S)$
can be embedded in the plane~$\mathbb R^2$ so that
\begin{enumerate}
\item
the inclusion map~$\Gamma_{x_*}(S)\hookrightarrow\mathbb R^2$ is proper;
\item
the cyclic order of half-edges at every vertex~$x$ of~$\Gamma_{x_*}(S)$ coincides
with the one induced by the orientation of~$\mathbb R^2$.
\end{enumerate}

Let~$U\subset\mathbb R^2$ be a closed regular neighborhood of~$\Gamma_{x_*}(S)$.
The number of topological ends of~$\Gamma_{x_*}(S)$ is the same as the number of
connected components of~$\partial U$, which are put into a correspondence
with sequences satisfying~\eqref{half-edge-seq-eq}, viewed up to a shift, as follows.

Let~$\beta$ be a connected component of~$\partial U$. Then the connected component
of~$U\setminus\Gamma_{x_*}(S)$ is a strip homeomorphic to~$(0,1]\times\mathbb R$
attached to~$\Gamma_{x_*}(S)$ along a path, which we denote by~$\widehat\beta$.
The decomposition of~$\widehat\beta$ into half-edges gives rise to a sequence
satisfying~\eqref{half-edge-seq-eq}. One can see that thus obtained correspondence
between connected components of~$\partial U$ and sequences satisfying~\eqref{half-edge-seq-eq}
and viewed up to a shift is again one-to-one.

The claim follows.
\end{proof}

\begin{proof}[Proof of Theorem~\ref{ergodicity-dim-th}]
For any leaf~$\gamma$ of the foliation~$\mathscr F$, we define~$\gamma^\vee$
as follows. Let~$x\in D$ be such that~$p(\gamma)\subset O_x(S)$.
If~$x\notin D_2$ we put~$\gamma^\vee=\gamma$. Otherwise,
we put~$\gamma^\vee=p^{-1}(O_x(S))\setminus\gamma$.

It follows from Lemma~\ref{number-of-leaves-lem} that~$\gamma^\vee$ is also a leaf of~$\mathscr F$
and~$(\gamma^\vee)^\vee=\gamma$.

If~$X\subset\Sigma$ is a union of leaves of~$\mathscr F$ we define~$X^\vee$ to
be~$\bigcup\gamma^\vee$, where the union is taken over all leaves~$\gamma$ of~$\mathscr F$
contained in~$X$.

Let~$E$ be an ergodic component of~$S$. Since both~$D_1$ and~$D_2$ are, clearly, invariant subsets
of~$D$, and the complement~$D\setminus(D_1\cup D_2)$ is negligible, we may assume~$E\subset D_1$
or~$E\subset D_2$.

According to Lemma~\ref{number-of-leaves-lem},
if~$E\subset D_1$, then the preimage of any orbit~$O_x(S)$ with~$x\in E$ consists of exactly one leaf
of~$\mathscr F$. Therefore, $p^{-1}(E)$ is an ergodic component of~$\mathscr F$.

If~$E\subset D_2$, then the preimage~$p^{-1}(E)$ is a union of leaves of~$\mathscr F$,
but not necessarily
a single ergodic component. Suppose that~$E_1\subset p^{-1}(E)$ is an ergodic component of~$\mathscr F$,
and let~$E_2=E_1^\vee$.

We may assume that, for any~$x\in E$, the preimage~$p^{-1}(O_x(S))$ contains at least one leaf
of~$\mathscr F$ lying in~$E_1$. Indeed, the subset~$p(E_1)\subset D$ cannot be negligible, therefore,
negligible must be the subset~$Y=\{x\in D:p^{-1}(O_x(S))\cap E_1=\varnothing\}$, and hence~$p^{-1}(Y)$
is also negligible. So, we can replace~$E_1$ by~$E_1\cup p^{-1}(Y)$.

By Lemma~\ref{number-of-leaves-lem} we have~$E_1\cup E_2=p^{-1}(E)$.

Thus, we see that the full preimage of any ergodic component of~$S$ is either a single ergodic component
of~$\mathscr F$ or consists of exactly two ergodic components of~$\mathscr F$. The claim follows.\end{proof}

From the proof of Theorem~\ref{ergodicity-dim-th} one can also see the following.

\begin{proposition}
If the subset~$D_2\subset D$ is negligible, then~$\dim\mathscr M(S)=\dim\mathscr M(\mathscr F)$.
In particular, if~$D_2\subset D$ is negligible and~$S$ is uniquely ergodic, then~$\mathscr F$
is also uniquely ergodic.
\end{proposition}

It remains an open question wether the unique ergodicity of~$\mathscr F$ imply the negligibility of~$D_2$.

We have seen in Proposition~\ref{non-uniquely-ergodic-prop} that
the transformation~$T^{\mathrm{AR}}_\lambda$ may not be uniquely ergodic
for~$\lambda\in\rg_{\mathrm{irr}}$. The genus of the respective
surface~$\Sigma_\lambda^{\mathrm{AR}}$ is three, so, in principle, one
might expect~$T^{\mathrm{AR}}_\lambda$ to have three ergodic components.
However, the following statement, which follows from Theorems~\ref{s-lambda-th}
and~\ref{ergodicity-dim-th} implies that this is impossible.

\begin{corollary}\label{two-measure-cor}
If~$\lambda\in\rg_{\mathrm{irr}}$, then the Arnoux--Rauzy transformation~$T_\lambda^{\mathrm{AR}}$
has at most two distinct ergodic Borel probability measures.
\end{corollary}

This corollary can be also proved directly, without a reference to
systems of isometries. This is done
as follows.

There is an involution~$\Psi$ on~$\mathscr M(\mathscr F_\lambda^{\mathrm{AR}})$
defined in terms of the parameters~$x_1,\ldots,x_6$
introduced in the proof of Theorem~\ref{typically-uniquely-ergodic-th}
by the permutation~$(12)(34)(56)$.
If~$\lambda\in\mathscr R_{\mathrm{irr}}$, then any measure~$\mu\in\mathscr
M(\mathscr F_\lambda^{\mathrm{AR}})$ invariant
under~$\Psi$ can be identified with an invariant measure
on the symbolic dynamical system~$(\Omega(\lambda),\mathscr S)$ introduced
in the proof of Theorem~\ref{ar-words-th}, where the unique
ergodicity of this system is also established.

Thus, the multiplicity of the eigenvalue~$1$ of~$\Psi$ is one, whereas
the multiplicity of~$-1$ cannot be larger. Hence,~$\dim\mathscr M(\mathscr F_\lambda^{\mathrm{AR}})\leqslant2$.

This means that the two invariant measures constructed in the
proof of Proposition~\ref{non-uniquely-ergodic-prop} for specific~$\lambda$ represent two ergodic
components of~$T^{\mathrm{AR}}_\lambda$, and the respective
ergodic components~$E_1,E_2$ of~$\mathscr F_\lambda$ satisfy~$E_2=E_1^\vee$ (see
the proof of Theorem~\ref{ergodicity-dim-th} for this notation).

\end{document}